\documentclass[a4paper,12pt]{amsart}
\usepackage[ansinew]{inputenc}
 
\usepackage{geometry}

\usepackage{amsmath}
\usepackage{amsfonts}
\usepackage{amsthm}
\usepackage{amssymb}
\usepackage{graphicx}
\usepackage{tikz}
\usepackage{multirow}
\usepackage{enumerate}

\makeindex

\usepackage{soul} 
\sodef\so{}{.14em}{.4em plus.1em minus .1em}{.4em plus.1em minus .1em} 

\usepackage{aliascnt}
\usepackage{hyperref}
\hypersetup{ 
    colorlinks=false,
    pdfborder={0 0 0},}
\newcommand{\nref}[1]{\hyperref[#1]{\ref*{#1}}}

\usepackage{array} 
\usepackage[T1]{fontenc} 
\newcommand{\subtile}[1] 
{
	\vspace{-0.3cm}
	\begin{center}
 		{{\textsc{#1}}}\\
	\end{center}
	\vspace{0.1cm}
}

\newcommand{\imageplain}[2] 
{
      {\includegraphics[scale=#1]{#2.png}}
}
\newcommand{\image}[2] 
{
      \begin{center} 
	\includegraphics[scale=#1]{#2.png}
      \end{center}
}
\newcommand{\hamburger}[4] 
{
  \thispagestyle{empty}
  \vspace*{-2cm}
  \vspace{0.5cm}
  \begin{center}
    \Large \bf
    #1
  \end{center}
  \vspace{0.5cm}
  \begin{center}	
    Simon Lentner \\
    Algebra and Number Theory, 
    University Hamburg,\\
    Bundesstra{\ss}e 55, D-20146 Hamburg \\
    \texttt{simon.lentner@uni-hamburg.de}
  \end{center}
  \vspace{0.5cm}

}

\makeatletter
\newcommand{\xleftrightarrow}[2][]{\ext@arrow 3359\leftrightarrowfill@{#1}{#2}}
\newcommand{\xdashrightarrow}[2][]{\ext@arrow 0359\rightarrowfill@@{#1}{#2}}
\newcommand{\xdashleftarrow}[2][]{\ext@arrow 3095\leftarrowfill@@{#1}{#2}}
\newcommand{\xdashleftrightarrow}[2][]{\ext@arrow
3359\leftrightarrowfill@@{#1}{#2}}
\def\rightarrowfill@@{\arrowfill@@\relax\relbar\rightarrow}
\def\leftarrowfill@@{\arrowfill@@\leftarrow\relbar\relax}
\def\leftrightarrowfill@@{\arrowfill@@\leftarrow\relbar\rightarrow}
\def\arrowfill@@#1#2#3#4{%
  $\m@th\thickmuskip0mu\medmuskip\thickmuskip\thinmuskip\thickmuskip
   \relax#4#1
   \xleaders\hbox{$#4#2$}\hfill
   #3$%
}
\makeatother

\allowdisplaybreaks[1]

\newcommand{\rank}{\mbox{rank}}
\renewcommand{\gcd}{\mbox{gcd}}
\newcommand{\lcm}{\mbox{lcm}}
\renewcommand{\mod}{\;\mathrm{mod}\;}

\newcommand{\gr}{\mbox{gr}}
\newcommand{\ord}{\mbox{ord}}

\newcommand{\id}{\mbox{id}}

\newcommand{\B}{\mathcal{B}} 	
\newcommand{\g}{\mathfrak{g}} 	
\renewcommand{\sl}{\mathfrak{sl}}	
\newcommand{\Z}{\mathbb{Z}}  	
\newcommand{\N}{\mathbb{N}}  	
\newcommand{\C}{\mathbb{C}}  	
\newcommand{\Q}{\mathbb{Q}}  	
\renewcommand{\S}{\mathbb{S}}  	

\newcommand{\CC}{\mathcal{C}}   

\renewcommand{\L}{\mathcal{L}} 
\renewcommand{\l}{\ell} 

\theoremstyle{plain}
\newtheorem{theorem}{Theorem}[section]
\newtheorem*{theoremX}{Theorem}

\newtheorem{example}[theorem]{Example}

\newtheorem{lemma}[theorem]{Lemma}

\newtheorem{problem}[theorem]{Problem}

\newtheorem{remark}[theorem]{Remark}

\begin{document}

\enlargethispage{10\baselineskip}
\hamburger{Quantum affine algebras at small root of unity}
{}{}{2014}
\thispagestyle{empty}

\begin{abstract}
  We study the Frobenius-Lusztig kernel for quantum affine algebras at
  root
  of unity of small orders that are usually excluded in literature. These cases
  are somewhat  degenerate and we find that the kernel is in fact mostly related
  to different affine  Lie algebras, some even of larger rank, that
  exceptionally  sit inside the quantum affine algebra. This continues the
  authors study for quantum  groups associated to finite-dimensional Lie
  algebras in [Len14c].
  \end{abstract}
  \makeatletter
  \@setabstract
  \makeatother

\tableofcontents
\vspace{-1cm}
{Partly supported by the DFG Priority Program 1388 ``Representation
theory''}

\newpage

\section{Introduction}

A quantum affine algebra $U_q(\g)$ is a Hopf algebra over the field of
rational functions $\C(q)$ and can be viewed as a deformation of the
universal enveloping of an \emph{affine Lie algebra} $\g$. It is a special case
of the Drinfel'd-Jimbo quantum group and the next logical step after quantum
groups associated to finite-dimensional Lie algebras $\g$. Lusztig has in
\cite{Lusz94} studied an integral form in this situation, which is a Hopf
algebra $U_q^{\Z[q,q^{-1}],\L}(\g)$ over the ring $\Z[q,q^{-1}]$, and one may
again perform a specialization to a specific value $q\in\C^\times$ and thus
obtain a complex Hopf algebra $U_q^{\L}(\g)$. For $q$ not a root of unity these
algebras behave similar to $U_q(\g)$ but for $q$ an $\ell$-th root of unity,
the algebra and its representation theory becomes significantly more
interesting. The theory of quantum affine algebras is much less developed than
the theory for finite Lie algebras $\g$ - Lusztig proves several structural
results in \cite{Lusz94} with certain restrictions on the order of $q$; among
others he defined a ``small quantum group'' $_Ru$ (which is now
infinite-dimensional) and a Frobenius-homomorphism to the ordinary universal
enveloping algebra of $\g$.\\
Notable subsequent results were a PBW-basis established in \cite{Beck94} and a
complete description of the representation theory of the Drinfel'd Jimbo quantum
group in \cite{CP95} and for the specialization in \cite{CP97}, both for $\g$ an
untwisted affine Lie algebra. Among others they prove a factorization theorem 
of representations into representations of the ordinary universal
enveloping algebra of $\g$ and representations with highest weight ``less then
$\ell$''. This is analogous to the respective results for finite-dimensional
$\g$ (and to the Steinberg factorization theorem for
Lie groups over finite fields) and closely related to the Frobenius
homomorphism. In \cite{CP98} the study was extended
to twisted affine Lie algebras. As an application, the authors mention
that e.g. the symmetry of the affine Toda field theories are governed by quantum
affine algebras and they voice the hope that their study could help understand
affine Toda theories with certain specific values of the coupling constant.\\ 

The aim of this article is to clarify the Hopf algebra structure of the
restricted specialization $U_q^{\L}(\g)$ to such values $q$ where Lusztig's
restrictions on $q$ are \emph{violated}. For these \emph{small roots of unity},
the quantum group severely degenerates. The author has in \cite{Len14c}
performed such a study already in the case of $\g$ a finite Lie algebra and
found a similar Frobenius homomorphism to a universal enveloping algebra. In
these degenerate cases, neither the kernel nor the image of the Frobenius
homomorphism are associated to the initial Lie algebra $\g$. As an application
we noted in \cite{Len14c} that e.g. the case $\g=B_n,\ell=4$ seem to be closely
related to the vertex algebra of $n$ symplectic fermions; here the kernel of
the Frobenius homomorphism is a small quantum group of type $A_1^{\times n}$ and
the image is the universal enveloping of $C_n=sp_{2n}$. It is to be expected
that a similar study for affine $\g$ would explain exceptional behaviour for
affine Toda field theories at certain small values of the coupling constants.\\

In this article we concentrate on the study of the kernel of the (yet-to-be
constructed) Frobenius homomorphism for affine $\g$. More precisely, we define a
suitable Hopf subalgebra $u_q^\L(\g)$ and describe its structure. It
turns out to be mostly governed by subsystems of the dual root system, while in
several exotic cases the quantum group itself changes into a quantum group
associated to a different Lie algebra of larger rank. In one set of cases it
even collapses to an infinite tower of quantum groups of finite Lie algebras.
Altogether we find:
\enlargethispage{2cm}
\begin{theoremX}[\nref{thm_main}]
  Let $\g$ be an affine Lie algebra and $q$ and $\ell$-th root of unity. We
  shall define a Hopf subalgebra $u_q^{\L}(\g)^+\subset
  U_q^\L(\g)^+$  with the following properties: 
  \begin{itemize}
    \item $u_q^\L(\g)^+$ consists of all degrees (roots) $\alpha$ with
    $\ell_\alpha\neq 1$. 
    \item Except for cases marked \emph{deaffinized}, $u_q^{\L}(\g)^+$
    is generated by primitives $E_{\alpha_i^{(0)}}$, spanning a braided vector
    space $M$ (in the deaffinized cases $u_q^{\L}(\g)^+$ is an infinite
    extension tower, see Section \nref{sec_A11})
    \item Lusztig's $_Ru\subset u_q^{\L}(\g)^+$ and equality only holds for
    \emph{trivial} and \emph{generic} cases.
    \item Except for cases marked \emph{deaffinized} and
    (possibly) \emph{exotic} $u_q^\L(\g)^+$ is coradically graded and hence maps
    onto the Nichols algebra $\B(M)=u_{q'}(\g^{(0)})^+$. For untwisted affine
    $\g$ we prove (and else conjecture) they are isomorphic.
    \item Except for cases marked \emph{deaffinized}, the $u_q^{\L}(\g)^+$
    fulfills $\ell_i\neq 1$ as well as Lusztig's non-degeneracy condition
    $\ell_{\alpha_i}\geq -a_{ij}+1$. Hence we have determined subalgebras on
    which Lusztig's theory in \cite{Lusz94} may be applied.
  \end{itemize}
\begin{center}\noindent
We explicitly describe the type $\g^{(0)},q'$ of $M,\B(M)$ as follows
\begin{tabular}{ll|lll}
  $\g$ & $\ell$ & $M$ & $q'$ & comment \\
  \hline\hline
  all & $\ell=1,2$ & $\{0\}$ & $q$ & trivial\\
  \hline
  $A_1^{(1)}$ & $\ell=4$ & $A_1^{2\times}$ & $q$ & {deaffinized}\\
  $B_n^{(1)},D_{n+1}^{(2)}$ & $\ell=4$ 
    & $A_1^{2n\times}$ & $q$ & short roots, deaffinized\\
  $C_n^{(1)},A_{2n-1}^{(2)}$ & $\ell=4$
    & $D_n^{(1)}$ & $q$ & short roots\\
  $F_4^{(1)},E_6^{(2)}$ & $\ell=4$
    & $D_4^{(1)}$ & $q$ & short roots\\
  $G_2^{(2)},D_4^{(3)}$ & $\ell=3,6$
    & $A_2^{(1)}$ & $q$ & short roots\\
\end{tabular}

(continues on next page)

\begin{tabular}{ll|lll}
  \hline
  $A_2^{(2)}$ & $\ell=4$
    & $A_1^{2\times}$ & $q$ & very short roots, deaffinized\\
  $A_2^{(1)}$ & $\ell=8$
    & $A_1^{2n\times}$ & $q$ & very short roots\\
  $A_{2n}^{(2)}$ & $\ell=4$
    & $A_1^{2n\times}$ & $q$ & very short roots, deaffinized\\
  $A_{2n}^{(2)}$ & $\ell=8$
    & $A_{2n-1}^{(2)}$ & $q$ &  not very long roots\\
  \hline
  $A^{(2)}_{2}$ & $3$ & $A_2^{(1)}$ & $q$ & exotic\\
  $A^{(2)}_{2}$ & $6$ & $A_2^{(1)}$ & $-q$ & exotic\\
  $A^{(2)}_{2n}$ & $3,6$ & $A^{(2)}_{2n}$ & $q$ & (pseudo-)exotic\\
  $G_2^{(1)}$ & $4$ & $A_3^{(1)}$ & $\bar{q}$ & exotic\\
  $D_4^{(3)}$ & $4$ & $D_4^{(1)}$ & $q,\bar{q},-1$ & exotic\\
\end{tabular}
\end{center}
All cases not included in the list are \emph{generic cases}
$\g^{(0)}=\g$ with $u_q^\L(\g)^+={_R}u^+$.
\end{theoremX}

We shall now discuss the approach and results of this article in more detail:\\

Lusztig defines in \cite{Lusz94} Chp. 36 the subalgebra $_Ru$ to be generated
by all $E_{\alpha_i},\;\ell_{\alpha_i}\neq 1$ and works under the additional
restriction 
$$\ell_{\alpha_j}\geq 2\;\Rightarrow \;
  \ell_{\alpha_i}\geq -a_{ij}+1\qquad (\ast)$$
(and $\g$ not containing odd cycles). Then he establishes a Frobenius
homomorphism with kernel $_Ru$. Note that in contrast in the finite case he had
in \cite{Lusz90b} Thm. 8.3 defined $u$ without restrictions on $q$ as being
generated by \emph{all} root vectors $E_\alpha$ with $\ell_\alpha\neq 1$, but
has refrained from doing so in the affine case by lack of root vectors (it
would be interesting to now use the results in \cite{Beck94}). Note that
frequently already in the finite case $u$ violating $\ell_{\alpha_i}\neq 1$ is
\emph{not} generated by simple root vectors.\\

The aim of this article is to nevertheless find a suitable subalgebra
$u_q^\L$  as in \cite{Lusz90b} and more
importantly calculate the structure of $u_q^\L$ in all cases violating either
$\ell_{\alpha_i}\neq 1$ or $(\ast)$. In the first case the type of $u_q^\L$ will
be determined by a subset of roots (mostly the short roots), which will be
characterized by a subsystem of the \emph{dual} root system. In the second
\emph{exotic cases} the root system severely changes (compare
the only finite example $G_2,\ell=4$).\\

In Section \nref{sec_subsystem} we determine the subsystems of the affine root
systems consisting of all roots divisible by a fixed integer. This extends
results in e.g. \cite{Car05} Prop. 8.13. in the case of a finite $\g$ and is
responsible for most of the root system data in the main theorem.\\

In Section \nref{sec_main} we formulate the main theorem and prove the general
statements. The final case-by-case analysis is performed in the remaining
sections \nref{sec_degenerate}-\nref{sec_A11}:\\

In Section \nref{sec_degenerate} we consider all \emph{degenerate cases} where
some simple roots fail $\ell_{\alpha_i}\neq 1$; this leaves only a subset of
generators $E_{\alpha_i^{(0)}}=E_{\alpha_i}$ for $u_q^\L(\g)$. We first give a
general approach to find more exceptionally primitive elements
$E_{\alpha_j^{(0)}}$ for small $q$. Namely, the Lusztig reflection operator is
still defined with respect to the root system of $\g$, whereas the ``true'' root
system generated by simple root vectors is now smaller. The images of such
``inappropriate'' reflections turn out to be new primitives. We then give a
generic argument that shows the dual subsystems found in Section
\nref{sec_subsystem} characterizes the quantum affine algebra generated by all
$E_{\alpha_i^{(0)}}$.\\

In Section \nref{sec_exotic} we turn to the \emph{exotic cases} where all
$\ell_{\alpha_i}\neq 1$, so all $E_{\alpha_i^{(0)}}=E_{\alpha_i}$, but the
second condition $(\ast)$ is violated. These cases are very interesting,
because the $E_{\alpha_i}$ only generate a (smaller) different quantum affine
algebra than expected, so one again has to add additional primitives and the
rank is now larger than $\g$. A similar phenomenon has been observer by the
author already in the case $\g=G_2,\ell=4$, where the corresponding
$u_q^\L(\g)^+$ is isomorphic to $u^\L_{\bar{q}}(A_3)^+$. Approach these cases
one-by-one and apply the theory of Nichols algebras to determine the subalgebra
generated by the $E_{\alpha_i}$. Then we find additional primitives (sometimes
using techniques from Section \nref{sec_degenerate}, some by guess-and-check),
until we finally account for all roots with $\ell_\alpha\neq 1$.\\

We finally turn to cases where Section \nref{sec_A11} returns one or
more copies of $A_1^{(1)},\ell=4$. One could accept this result, but it again
violates $(\ast)$, so we study it a bit further. Surprisingly this is the most
misbehaved case: The $E_{\alpha_0},E_{\alpha_1}$ span a braided vector space of
type $A_1\times A_1$. They generate a (quasi-) classical universal enveloping of
type $A_2$ (so there is so-called nontrivial liftings). The entire algebra
$u_q(A_1^{(1)})$ can be described as an infinite tower of extensions by other
$A_2$ algebras. We can only make such observations by using Drinfel'd
alternative generating system as determined by \cite{Beck94}, which views the
algebra as an explicit affinization (not a mere Cartan matrix).\\

We close in section \nref{sec_questions} by stating some open questions that
were out of the scope of this paper.

\section{Preliminaries}
\subsection{Affine Lie algebras}\label{subsec_affine}

Our exposition is largely from \cite{Kac84} Sec. 1. Affine Lie
algebras are characterized by the fact that the Cartan matrix is positive
semidefinite and there is a unique isotropic root $\delta\Z$ of length $0$,
which is clearly not in the Weyl group orbit of any simple roots. Proper
parabolic
subsystems always correspond to finite dimensional Lie algebras and there is a
common choice of such a parabolic subsystem $\{\alpha_1,\ldots
\alpha_{{\mathrm rank}(\g)-1}\}$ corresponding to a finite root system
$\bar{\Delta}$ which
is extended by an additional simple root $\alpha_0$. Affine Lie algebras are
classified: The so-called \emph{twisted affine Lie algebras}
$X_n^{(1)}=\hat{X}_n$ are obtained by extending a finite Lie algebra of type
$X_n$ by the negative highest root $\alpha_0$; they can be realized by centrally
extending the loop algebra $X_n\otimes \C[t,t^{-1}]$. The other so-called
\emph{twisted affine Lie algebras} can be realized similarly by an extension
involving an outer automorphism of $X_n$; this is what Kac's notation
$X_n^{(2)},D_4^{(3)}$ refers to, which is well established especially in
physics. Other authors such as Carter, Fuchs, etc. denote the twisted affine Lie
algebras in a way that emphasizes the similar Weyl group and (equivalently) how
a parabolic finite Lie algebra with root system $\bar{\Delta}$ is extended by a
different $\alpha_0$ then in the untwisted case. For example, both
$G_2^{(1)},D_4^{(3)}$ have affine Weyl groups $G_2$ and can be obtained from
extending $\bar{\Delta}=G_2$. The author would prefer the second notation, but
sticks with the more common one.\\ 

Another less direct construction the author finds convenient (and may not be
new) is to obtain affine Lie algebras from simply-laced untwisted Lie algebras
by the \emph{folding} procedure: Let $\g$ be a Lie algebra and $f$ an
automorphism of the Dynkin diagram, then we may consider the Lie subalgebra
$\g^f$ fixed by $f$. Note this is \emph{not} the same as the twisted
realizations using an automorphism of the finite root system!\\

Note also that as in the finite case one may form (in
non-simply-laced cases) the dual root system consisting of rescaled coroots
$\alpha_i^\vee=\frac{2}{(\alpha_i,\alpha_i)}\alpha_i$, which switches long and
short roots. We summarize all Dynkin diagrams and then all properties:
\enlargethispage{3cm}
\begin{center}
\begin{tabular}{m{2.5cm}m{14cm}}
$A_1^{(1)}$ & \imageplain{.2}{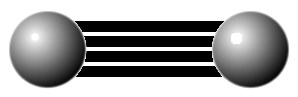}\\
$A_{n}^{(1)},\;n\geq 2$ & \imageplain{.2}{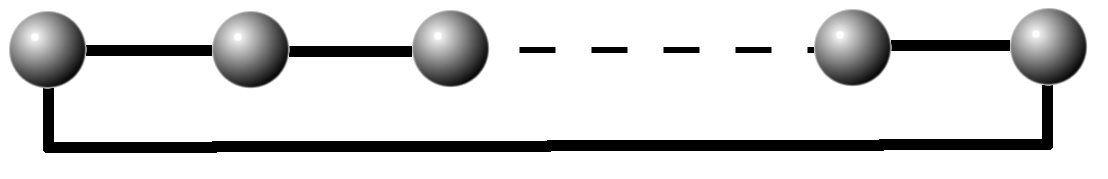}\\
$B_{n}^{(1)},\;n\geq 3$ & \imageplain{.2}{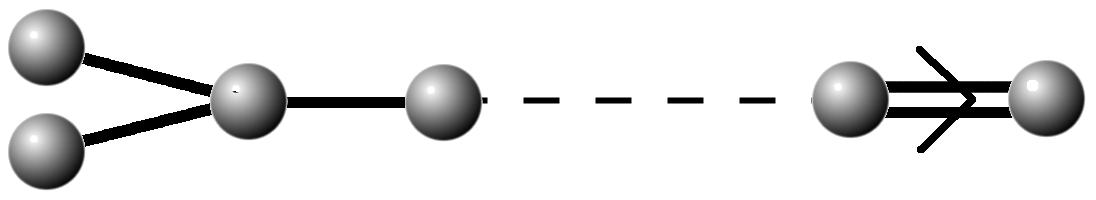}\\
$C_{n}^{(1)},\;n\geq 2$ & \imageplain{.2}{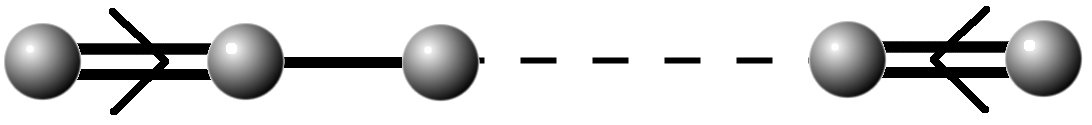}\\
$D_{n}^{(1)},\;n\geq 4$ & \imageplain{.2}{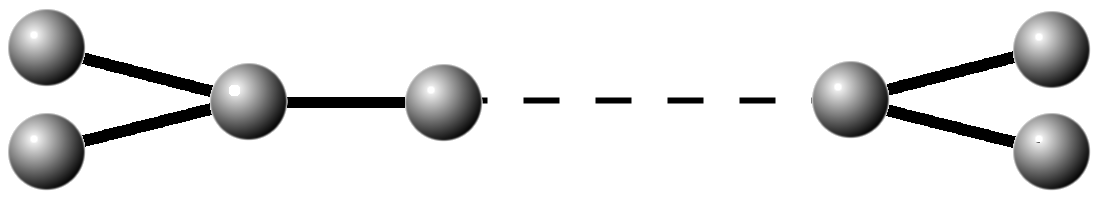}\\
$E_6^{(1)}$ & \imageplain{.2}{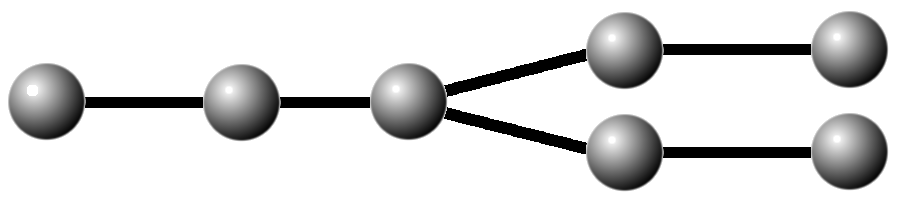}\\
$E_7^{(1)}$ & \imageplain{.2}{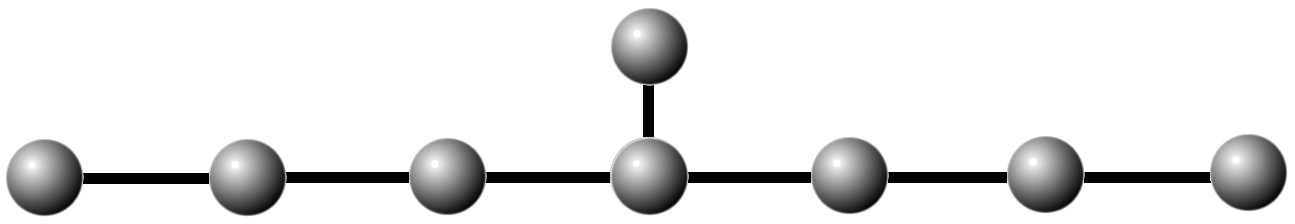}\\
$E_8^{(1)}$ & \imageplain{.2}{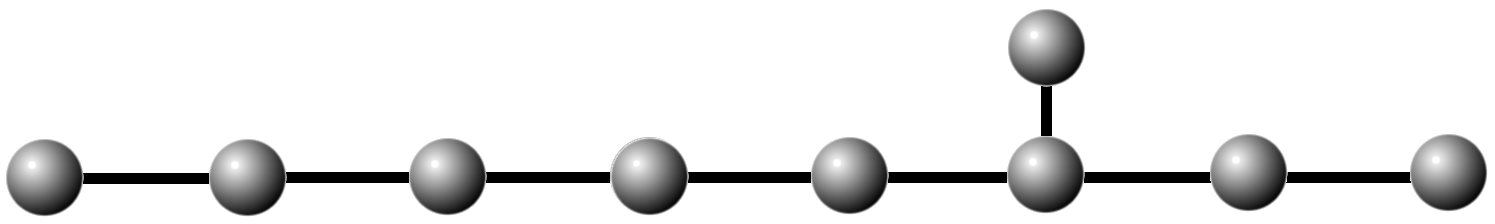}\\
$F_4^{(1)}$ & \imageplain{.2}{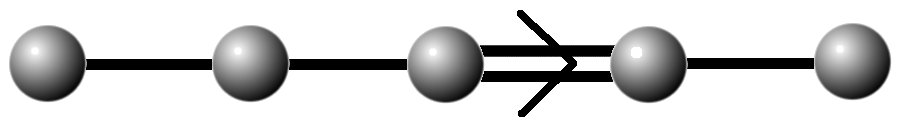}\\
$G_2^{(1)}$ & \imageplain{.2}{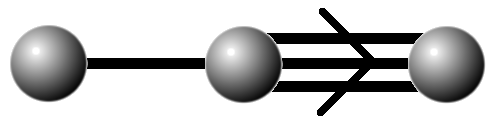}\\
\hline\vspace{.2cm}\\
$D_{n+1}^{(2)},\;n\geq 2$ & \imageplain{.2}{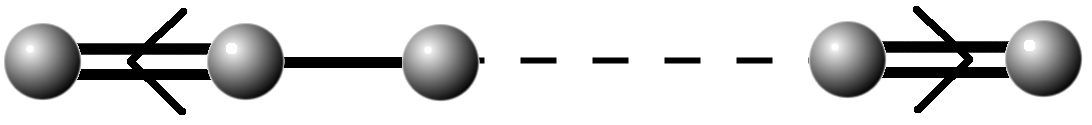}\\
$A_{2n-1}^{(2)},\;n\geq 3$ & \imageplain{.2}{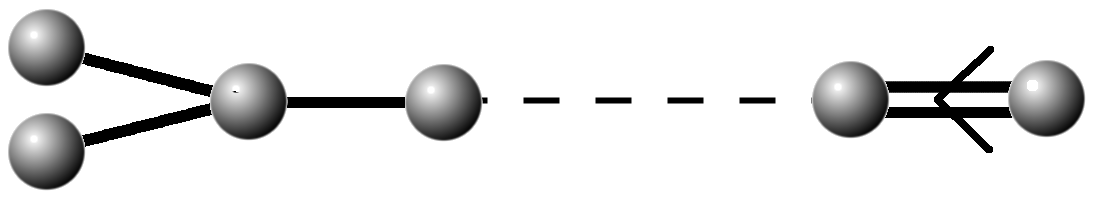}\\
$E_6^{(2)}$ & \imageplain{.2}{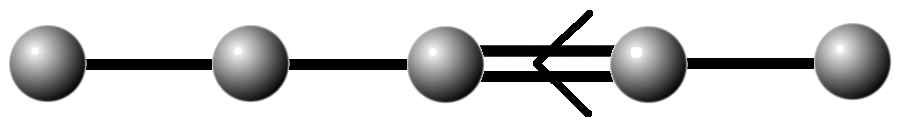}\\
$D_4^{(3)}$ & \imageplain{.2}{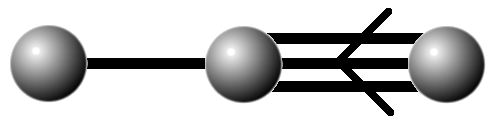}\\
\hline\vspace{.2cm}\\
$A_2^{(2)}$ & \imageplain{.2}{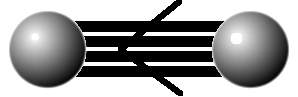}\\
$A_{2n}^{(2)},\;n\geq 2$ & \imageplain{.2}{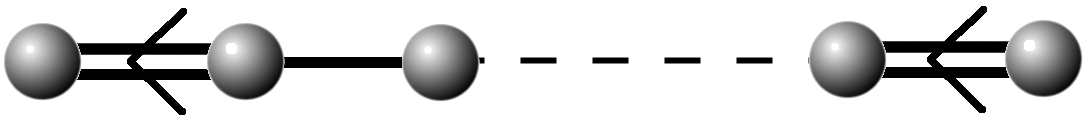}\\
\end{tabular}
\end{center} 

\renewcommand{\thefootnote}{\fnsymbol{footnote}}
\begin{center}
\begin{tabular}{lll|l|l}
$\g$ & rank & $\bar{\Delta}$ & Dual root system $\Delta^\vee\quad$ & Folding
via \\
 \hline
$A_1^{(1)}$ & $2$ & $A_1$ &  & \footnotemark[2]
$A_3^{(1)}\;/\;\Z_2\times\Z_2$ \\
$A_n^{(1)}$ & $n+1$ & $A_n$ & 
  & \footnotemark[2] $A_{k(n+1)-1}^{(1)}\;/\;\Z_k\quad \forall_k$ \\
$B_n^{(1)}$ & $n+1$ & $B_n$ & $A_{2n-1}^{(2)}$ & $D_{n+1}^{(1)}\;/\;\Z_2$ \\
$C_n^{(1)}$ & $n+1$ & $C_n$ & $D_{n+1}^{(2)}$ 
  & \footnotemark[2] $A_{2n-1}^{(1)}\;/\;\Z_2$ \\
$D_n^{(1)}$ & $n+1$ & $D_n$ & & \\
$E_6^{(1)}$ & $7$ & $E_6$ & &  \\
$E_7^{(1)}$ & $8$ & $E_7$ & &  \\
$E_8^{(1)}$ & $9$ & $E_8$ & &  \\
$F_4^{(1)}$ & $5$ & $F_4$ & $E_{6}^{(2)}$ & $E_6^{(1)}\;/\;\Z_2$ \\
$G_2^{(1)}$ & $3$ & $G_2$ & $D_{4}^{(3)}$ & $D_4^{(1)}\;/\;\S_3$ \\
\hline
$D_{n+1}^{(2)}$ & $n+1$ & $B_{n}$ & $C_{n}^{(1)}$ 
  & $D_{n+2}^{(1)}\;/\;\Z_2\times\Z_2$\\
$A_{2n-1}^{(2)}$ & $n+1$ & $C_{n}$ & $B_{n}^{(1)}$ 
  & $D_{2n}^{(1)}\;/\;\Z_2$ \\
$E_6^{(2)}$ & $5$ & $F_4$ & $F_4^{(1)}$  & $E_7^{(1)}\;/\;\Z_2$ \\
$D_4^{(3)}$ & $3$ & $G_2$ & $G_2^{(1)}$ & $E_6^{(1)}\;/\;\S_3$ \\
\hline
$A_2^{(2)}$ & $2$ & $A_1$ & selfdual & $D_4^{(1)}\;/\;\S_4$ \\
$A_{2n}^{(2)}$ & $n+1$ & $C_{n}$ & selfdual 
  & $D_{2n+2}^{(1)}\;/\;(\Z_2\times\Z_2)\rtimes \Z_2$ \\
\end{tabular}
\end{center} 
\footnotetext[2]{A very remarkable fact is that the simply-laced $A_{n}^{(1)}$
of rank $n+1$ can be folded from $A_{k(n+1)-1}^{(1)}$
of rank $k(n+1)$ for any $k$ via a rotation of order $k$. As a consequence,
there is also an infinite series of folding for $A_1^{(1)}$ as well as an
infinite series of folding $C_{n}^{(1)}$ of rank $n+1$ from $C_{2n}^{(1)}$ of
rank $2(n+1)-1$.}

Two important numbers associated to the affine Lie algebra is the
superscript number $k=(1),(2),(3)$ and the number $a_0=1$ for all cases except
$A_2^{(2)},A_{2n}^{(2)}$ have $a_0=2$. We will frequently distinguish the cases
$a_0k=1,2,3,4$. We denote $n=\rank(\bar{\Delta})=\rank(\Delta)-1$. We summarize
the description of the root system of the affine Lie algebras from \cite{Kac84}
Sec. 1.4:\\

The \emph{isotropic roots} (i.e. length $0$ and hence not in the Weyl orbit of
a simple root) are 
$$\Delta^{im}=\delta\Z\;\backslash\;\{0\},\qquad \delta:=a_0\alpha_0+\theta$$
where $\theta$ is the highest root of $\bar{\Delta}$ for $a_0k=1,4$ and the
highest short root for $a_0k=2,3$. Hence it is often convenient to draw
$\Delta$ in the root system $\Delta$ by projecting $\delta=0$ and hence drawing
$\alpha_0=-a_0^{-1}\theta$.\\
The \emph{multiplicity} of each root $m\delta$ is
$n=\rank(\bar{\Delta})$ with the following exceptions:
\begin{center}
  \begin{tabular}{l|llll}
  $\g$			& $A_{2n-1}^{(2)}\quad$ & $D_{n+1}^{(2)}\quad$ 
			& $E_6^{(2)}\quad$ & $D_4^{(3)}\quad$\\
  \hline
  $m$			& $2\nmid m$ & $2\nmid m$ & $2\nmid m$ & $3\nmid m$ \\
  mult$(m\delta)$	& $n-1$	     & $1$ & $2$ & $1$ \\
  \end{tabular}
\end{center}~\\
  
The \emph{real roots} (i.e. in the Weyl orbit) $\Delta^{re}$ are as follows
\begin{center}
  \begin{tabular}{ll}
  $a_0k=1$ & $\Delta^{re}=\bar{\Delta}+\delta\Z$ \\
  $a_0k=2,3\quad$ & $\Delta^{re}=\left(\bar{\Delta}^{short}+\delta\Z\right)
  \;\cup\; \left(\bar{\Delta}^{long}+\delta k\Z\right)$\\ 
  $a_0k=4\quad$ & $\Delta^{re}=\left(\bar{\Delta}^{short}+\delta\Z\right)
  \;\cup\; \left(\bar{\Delta}^{long}+\delta k\Z\right)
  \;\cup\; \left(\frac{1}{2}(\bar{\Delta}^{long}+\delta)+\delta \Z\right)$\\ 
  \end{tabular}
\end{center}
and all real roots have \emph{multiplicity} $1$.

\subsection{Affine quantum groups}

In \cite{Lusz94} Sec. 1.2 Lusztig defines a Hopf algebra $f$ over $\Q(q)$
associated to a Cartan datum. In modern terminology, consider the category
$\CC^{\Q(q)}$ of $\N^\Pi$-graded $\Q(q)$-vector spaces $V$ with braiding on
homogeneous elements $x_\alpha\otimes y_\beta \mapsto q^{(\alpha,\beta)}$. Let
$V\in\CC^{\Q(q)}$ be the vector space spanned by symbols $E_{\alpha_i}$.

Let $f'=TV$ be the tensor algebra which becomes a Hopf algebra in the braided
category $\CC^{\Q(q)}$ by defining $V$ to be primitive, i.e.
$\Delta(E_{\alpha_i})=1\otimes  E_{\alpha_i}+E_{\alpha_i}\otimes 1$. Then,
there is a unique symmetric Hopf pairing defined by 
$$(E_{\alpha_i},E_{\alpha_j})=\delta_{i,j}(1-q^{-(\alpha_i,\alpha_j)})^{-1}$$
Now let $I$ be the radical of the Hopf pairing, then define the Hopf algebra
$$f:=f'/\mathfrak{I}$$
In Sec. 36 an integral form, i.e. a Hopf algebra $_\mathfrak{A}f$ over the ring 
$\mathfrak{A}:=\Z[q,q^{-1}]$ is defined and then the specialization $_Rf$ 
restricted to a specific value $q\in\C^\times$ via $\otimes_\mathfrak{A} R$
where we may take $R=\C_q$ the field $\C$ with $q$ acting by the specified
value. {\bf Note this is only possible for good values}
$q^{(\alpha_i,\alpha_i)}\neq 1$ and leads to what we today call the Nichols
algebra $\B(V)$. These restrictions are always in place throughout the reminder
of \cite{Lusz94}. \\

On the other hand in Sec. 3 Lusztig proceeds, without any restrictions on $q$,
as in \cite{Lusz90a}\cite{Lusz90b} for finite Cartan datum. He takes the
Drinfel'd-Jimbo quantum group $U$ over $\Q(q)$, defines an integral form
(restricted form) $_\mathfrak{A}U$ over the ring $\mathfrak{A}:=\Z[q,q^{-1}]$
and performs again specialization $_RU$. We again take $R=\C_q$ and denote this
Hopf algebra over $\C$ by $U_q^\L(\g)$. Moreover, the Borel part
$U_q^{\L}(\g)^+$ is again a Hopf algebra in the Yetter-Drinfel'd modules over
the root lattice $\Lambda_R$ of $\g$. Compare the authors introduction in
\cite{Len14c}.\\

From this point the case of affine Lie algebras start to exhibit an incomplete
picture. Lusztig does not prove that reflection provides a PBW-basis, see
Lusztig's respective question in \cite{Lusz94} Sec. 40.2.3, although he gives a
sketch how such a fact might be proven, using an infinite sequence of Weyl
group elements of ascending length. He establishes a Frobenius homomorphism if
the root of unity fulfills the restrictions, namely $q^(\alpha_i,\alpha_i)\neq
1$ and the Lie algebra Dynkin diagram does not contain odd cycles. He shows that
under these restrictions on $q$ the kernel of the Frobenius homomorphism, which
we call today \emph{Frobenius-Lusztig kernel}, is generated by all
$E_{\alpha_i}$ with $q^{(\alpha_i,\alpha_i)}$.\\

Note that in the completely worked out case in \cite{Lusz90a}\cite{Lusz90b}
for finite root systems the Frobenius-Lusztig kernel is shown to be
generated by \emph{all} root vectors $E_\alpha,q^{(\alpha,\alpha)}\neq 1$. As
the author has worked out in \cite{Len14c}, this Hopf algebra is for arbitrary
$q$ associated to a different Lie algebra and in one exotic case $G_2,q=\pm i$
not even generated just by the simple root vectors $E_{\alpha_i}$. The author
also provides a different proof for the Frobenius homomorphism that includes
arbitrary $q$.\\

The aim of this article is to consider all the affine Lie algebras $\g$ and
the roots of unity $q$ that \emph{violate} the condition
$q^{(\alpha_i,\alpha_i)}\neq 1$ and realize the affine Frobenius-Lusztig kernel
(of usually different Lie type) inside $U_q^\L(\g)^+$ that corresponds precisely
to the set of all roots $q^{(\alpha,\alpha)}\neq 1$. 

\section{The subsystem of long roots}\label{sec_subsystem}

A natural question for any root system $\Delta$ is the subsystem of
roots divisible by a fixed number $t\in\N$, though for affine Lie algebras the
author has not found this question addressed in literature explicitly. The aim
of this section is to determine this subsystem. We note already at this point,
that the structure of the Frobenius-Lusztig kernel will often (but not always)
be described by the set of \emph{short} roots, hence the set of long roots in
the dual root system. In the finite case treated by the author in \cite{Len14c}
these were in fact all but a single exotic case $G_2,q=\pm i$.\\

\begin{lemma}\label{lm_DeltaT}
  Let $\Delta$ be a root system and $t\in \N$, then the subset
  $$\Delta^t:=\{\alpha\in\Delta\;|\;t|(\alpha,\alpha)\}$$
  is a root system. Note that $\Delta^2=\Delta$ and arbitrary $\Delta^t$ may be
  empty. By convention, isotropic roots are contained in any $\Delta^t$.
\end{lemma}
\begin{proof}
  We only have to prove that for $\alpha,\beta\in\Delta^t$ and
  $\alpha+\beta\in\Delta$ we have $\alpha+\beta\in\Delta^t$:
  $$(\alpha+\beta,\alpha+\beta)=(\alpha,\alpha)
  +(\beta,\beta)+\frac{2(\alpha,\beta)}{(\alpha,\alpha)}(\alpha,\alpha)
  \in t\Z$$
  since the Cartan matrix $\frac{2(\alpha,\beta)}{(\alpha,\alpha)}\in\Z$. 
\end{proof}

We consider the cases $\Delta^t=\Delta$ \emph{generic} and all cases with
$\Delta^t$ only isotropic roots \emph{trivial}. Hence for simply-laced root
systems there are only generic and trivial cases.\\

\subsection{Finite root systems}

For completeness and later use we first we give the table with all
nontrivial/nongeneric cases for finite root systems $\Delta$, which has been
implicitly already used in \cite{Len14c} and has can be found in \cite{Car05}
Prop. 8.13. , including the new set $\Pi^t$ of
positive simple roots $\alpha_i^{t}$ for $\Delta^t$:
\begin{lemma}\label{lm_longFinite}
  Let $\Delta$ be a connected finite root system, then all $\Delta^t$
  are equal to $\Delta$ or empty except the following
  cases:\\

  \begin{center}
  \begin{tabular}{ll|ll}
  $\Delta$ & $t$ & $\Delta^t$ & $\Pi^t$\\
  \hline
  $B_n$ & $4$ & $D_n$ & $\alpha_1,\alpha_2,\ldots,\alpha_{n-1},
    \alpha_{n-1}+2\alpha_n$\\
  $C_n$ & $4$ & $A_1^{\times n}$ & $\alpha_n,\alpha_n+2\alpha_{n-1},
    \alpha_n+2\alpha_{n-1}+2\alpha_{n-2},\dots$\\
  $F_4$ & $4$ & $D_4$ & $\alpha_1,\alpha_2,\alpha_2+2\alpha_3,
    \alpha_2+2\alpha_3+2\alpha_4$\\
  $G_2$ & $3,6$ & $A_2$ & $\alpha_1,\alpha_1+3\alpha_2$ 
  \end{tabular}
  \end{center}
  Note that our choice of $\Pi^t$ is minimal in the sense that any
  $\alpha_i^t\in\Pi^t$ contains precisely one simple root $\Delta^t\cap \Pi$ and
  only with multiplicity one. 
\end{lemma}
\begin{proof}
  Since all root lengths are $2,4$ resp. $2,6$ we only have to consider the
  cases in the statement. We proceed case-by-case and first show that the
  simple roots $\alpha_i^t$ indeed have the claimed new Cartan matrix; then we
  show by  counting that we have indeed found all new roots.
  \begin{itemize}
   \item Let $\Delta=B_n$ and $t=4$. The simple roots
      $\alpha_i^t:=\alpha_i,\;1\leq i\leq n-1$ are
      indeed long and have a Cartan matrix of type $A_{n-1}$.
      For the root $\alpha_n^t:=\alpha_{n-1}+2\alpha_n$ we check
      \begin{align*}
	(\alpha_n^t,\alpha_n^t)
	&=(\alpha_{n-1},\alpha_{n-1})+4(\alpha_n,\alpha_{n-1})
	+4(\alpha_n,\alpha_n)=4-8+8=4\\
	(\alpha_n^t,\alpha_{n-1}^t)
	&=(\alpha_{n-1},\alpha_{n-1})+2(\alpha_n,\alpha_{n-1})
	=4-4=0\\
	(\alpha_n^t,\alpha_{n-2}^t)
	&=(\alpha_{n-1},\alpha_{n-2})+2(\alpha_n,\alpha_{n-2})=-2\\
	(\alpha_n^t,\alpha_{i<n-2}^t)&=0
      \end{align*}
      This is the Cartan matrix of $D_n$ with center node $\alpha_{n-2}$ as
      claimed. It is known that $B_n$ has $2n^2$ roots and $2n(n-1)$ long roots.
      Since $D_n$ has also $2n(n-1)$ roots we see that $\Delta^t=D_n$.
    \item Let $\Delta=C_n$ and $t=4$. We check that all roots
      $\alpha_i^t=\alpha_n+2\alpha_n\cdots+2\alpha_{n-i+1}$ by induction:
      \begin{align*}
	(\alpha_1^t,\alpha_1^t)
	&=(\alpha_n,\alpha_n)=4\\
	(\alpha_{i+1}^t,\alpha_{i+1}^t)
	&=(\alpha_{i}^t,\alpha_{i}^t)
	+4(\alpha_{n-i},\alpha_{i}^t)+4(\alpha_i,\alpha_i)\\
	&=(\alpha_{i}^t,\alpha_{i}^t)
	+4(\alpha_{n-i},\alpha_n+\cdots+2\alpha_{n-i+1})+4(\alpha_i,\alpha_i)\\
	&=4-8+8=4
      \end{align*}
      Next we convince ourselves that all $\alpha_i^t$ have a $A_1^{\times n}$
      Cartan matrix i.e. are orthogonal (this does not mean
      $\alpha_i^t+\alpha_j^t\not\in\Delta$), let $i>j$: 
      \begin{align*}
	(\alpha_i^t,\alpha_j^t)
	&=(\alpha_j^t,\alpha_j^t)+(2\alpha_{n-j}+\cdots
	+2\alpha{n-i+1},\alpha_j^t)
	=4-4=0
      \end{align*}
      Hence the subsystem generated by the $\alpha_i^t$ is indeed of type
      $A_1^{\times n}$. It is known that $C_n$ has $2n^2$ roots and $2n$ long
      roots. Since $A_1^{\times n}$ has also $2n$ roots we have
      $\Delta^t=A_1^{\times n}$.
      \item Let $\Delta=F_4$ and $t=4$. We again calculate the Cartan matrix of
      $\alpha_i^t:=\alpha_1,\alpha_2,\alpha_2+2\alpha_3,
      \alpha_2+2\alpha_3+2\alpha_4$ as follows:
      \begin{align*}
	(\alpha_1^t,\alpha_1^t)
	&=(\alpha_1,\alpha_1)=4\\
	(\alpha_1^t,\alpha_2^t)
	&=(\alpha_1,\alpha_2)=-2\\
	(\alpha_1^t,\alpha_3^t)
	&=(\alpha_1,\alpha_2+2\alpha_3)=-2\\
	(\alpha_1^t,\alpha_2^t)
	&=(\alpha_1,\alpha_2+2\alpha_3+2\alpha_4)=-2\\
	(\alpha_2^t,\alpha_2^t)
	&=(\alpha_2,\alpha_2)=4\\
	(\alpha_2^t,\alpha_3^t)
	&=(\alpha_2,\alpha_2+2\alpha_3)=4-4=0\\
	(\alpha_2^t,\alpha_4^t)
	&=(\alpha_2,\alpha_2+2\alpha_3+2\alpha_4)=0\\
	(\alpha_3^t,\alpha_3^t)
	&=(\alpha_2+2\alpha_3,\alpha_2+2\alpha_3)\\
	&=(\alpha_2,\alpha_2)+4(\alpha_2,\alpha_3)+4(\alpha_3,\alpha_3)
	=4-8+8=4\\
	(\alpha_3^t,\alpha_4^t)
	&=(\alpha_3^t,\alpha_3^t)+(\alpha_2+2\alpha_3,2\alpha_4)
	=4-4=0\\
	(\alpha_4^t,\alpha_4^t)
	&=(\alpha_3^t,\alpha_3^t)+2(\alpha_2+2\alpha_3,2\alpha_4)+4(\alpha_4,
	\alpha_4)=4-8+8=4
      \end{align*}
      Hence the set of roots $\alpha_i^t$ generate a (rescaled) subsystem of
      type $D_4$ with center node $\alpha_1^t$ as claimed. It is known
      that $F_4$ has $48$ roots and $24$ long roots. Since $D_4$ has also
      $2n(n-1)=24$ roots we have $\Delta^t=D_4$.
      \item Let $\Delta=G_2$ and $t=3,6$. We again calculate the Cartan matrix
      of the short roots $\alpha_i^t=\alpha_1,\alpha_1+3\alpha_2$ as follows:
      \begin{align*}
	(\alpha_1^t,\alpha_1^t)
	&=(\alpha_1,\alpha_1)=6\\
	(\alpha_1^t,\alpha_2^t)
	&=(\alpha_1,\alpha_1)+3(\alpha_1,\alpha_2)
	=6-9=-3\\
	(\alpha_2^t,\alpha_2^t)
	&=(\alpha_1,\alpha_1)+6(\alpha_1,\alpha_2)+9(\alpha_2,\alpha_2)
	=6-18+18=6
      \end{align*}
      Hence the set of roots $\alpha_i^t$ generate a (rescaled) subsystem of
      type $A_2$ as claimed. It is known that $G_2$ has $6$ roots and $3$ long
      roots. Since $A_2$ has also $3$ roots we have $\Delta^t=A_2$.
  \end{itemize}
\end{proof}

\newpage
\subsection{Affine root systems}

We now determine the nontrivial/nongeneric cases for the affine Lie algebras
without considering the isotropic roots. 

\begin{theorem}\label{thm_longAffine}
    Let $\Delta$ be a connected affine root system, then all
    $\Delta^t$ are equal to $\Delta$ or only consist of isotropic
    roots except the following cases:\\
  
\begin{center}
  \begin{tabular}{ll|ll}
  $\Delta$ & $t$ & $\Delta^t$ & $\Pi^t$\\
  \hline
  $B_n^{(1)}$ & $4$ & $D_n^{(1)}$ 
  & $\alpha_0,\alpha_1,\alpha_2,\cdots,\alpha_{n-1},
    \alpha_{n-1}+2\alpha_n$\\
  $C_n^{(1)}$ & $4$ & $\left(A_1^{(1)}\right)^{\times n}$
  & $\alpha_n,\alpha_n+2\alpha_{n-1},
    \alpha_n+2\alpha_{n-1}+2\alpha_{n-2},\cdots$\\
  &&& $\alpha_0,\alpha_0+2\alpha_{1},
    \alpha_0+2\alpha_{1}+2\alpha_{2},\ldots$\\
  $F_4^{(1)}$ & $4$ & $D_4^{(1)}$ 
    & $\alpha_0,\alpha_1,\alpha_2,\alpha_2+2\alpha_3,
    \alpha_2+2\alpha_3+2\alpha_4$\\
  $G_2^{(1)}$ & $3,6$ & $A_2^{(1)}$ & $\alpha_0,\alpha_1,\alpha_1+3\alpha_2$ \\
  \hline
  $D_{n+1}^{(2)}$ & $4$ & $D_n^{(1)}$ &
    $\alpha_0',\alpha_1,\alpha_2,\cdots,\alpha_{n-1},\alpha_{n-1}+2\alpha_n$\\
  &&&$\quad\alpha_0':=2\alpha_0+\alpha_1$\\
  $A_{2n-1}^{(2)}$ & $4$ & $\left(A_1^{(1)}\right)^{\times n}$
  & $\alpha_n,\alpha_n+2\alpha_{n-1},
    \alpha_n+2\alpha_{n-1}+2\alpha_{n-2},\ldots$\\
  &&& $\alpha_0',\alpha_0'+2\alpha_{1},
    \alpha_0'+2\alpha_{1}+2\alpha_{2},\ldots$\\
  &&&$\quad\alpha_0':=2\alpha_0+2\alpha_2+2\alpha_3+\cdots+2\alpha_{n-1}
    +\alpha_n$\\
  $E_6^{(2)}$ & $4$ & $D_4^{(1)}$ 
  & $\alpha_0',\alpha_4,\alpha_3,\alpha_3+2\alpha_2,
    \alpha_3+2\alpha_2+2\alpha_1$\\
  &&&$\quad\alpha_0':=2\alpha_0+2\alpha_1+2\alpha_2+\alpha_3$\\
  $D_4^{(3)}$ & $3,6$ & $A_2^{(1)}$ & $\alpha_0',\alpha_2,\alpha_2+3\alpha_1$\\
  &&&$\quad\alpha_0':=3\alpha_0+3\alpha_1+\alpha_2$\\
  \hline
  $A_2^{(2)}$ & $4,8$ & $A_1^{(1)}$ & $4\alpha_0+\alpha_1,\alpha_1$\\
  $A_{2n}^{(2)}$ & $4$ & $A_{2n-1}^{(2)}$ 
    & $\alpha_0',\alpha_1,\ldots,\alpha_n$\\
  &&&$\quad \alpha_0':=2\alpha_0+\alpha_1$\\
  $A_{2n}^{(2)}$ & $8$ & $\left(A_1^{(1)}\right)^{\times n}$ 
  & $\alpha_n,\alpha_n+2\alpha_{n-1},
    \alpha_n+2\alpha_{n-1}+2\alpha_{n-2},\ldots$\\
  &&& $\alpha_0'',\alpha_0''+2\alpha_{1},
    \alpha_0''+2\alpha_{1}+2\alpha_{2},\ldots$\\
  &&&$\quad\alpha_0'':=2\alpha_0'+2\alpha_2+2\alpha_3+\cdots+2\alpha_{n-1}
    +\alpha_n$\\  
  &&&$\quad \alpha_0':=2\alpha_0+\alpha_1$\\
  \end{tabular}
  \end{center}
  Note that our choice of $\Pi^t$ is minimal in the sense that any
  $\alpha_i^t\in\Pi^t$ contains precisely one simple root $\Delta^t\cap \Pi$ and
  only with multiplicity one.
\end{theorem}
\begin{remark}
  With the exception of $A_{2n}^{(2)}$ the subsystem $\Delta^t$ turns out to
  be the affinization of the respective subsystem of the finite root system
  $\bar{\Delta}^t$ (this is more serious for the disconnected cases in the
  second and last row, where $\Delta^t$ has higher rank $2n$). Also it coincides
  whenever $\bar{\Delta}$ coincides. It would be nice to find a more systematic
  reason for this. Our proof uses a nice inclusion between the respective root
  systems, which may be known.
\end{remark}
\begin{proof}
  All root lengths are $2,4$ resp. $2,6$ except $2,8$ for 
  $A_2^{(2)}$ and $2,4,8$ for $A_{2n}^{(2)}$, hence we only have to consider the
  cases in the statement. We proceed similarly as (and using heavily) Lemma
  \nref{lm_longFinite}: Namely, we ``guess'' a set of positive simple roots
  $\Pi^t\subset \Delta^t$, calculate the root system generated by $\Pi^t$ and
  compare to the set $\Delta^t$ to show we indeed have found all.
  The set $\Delta^t$ can be easily read off the explicit form of
  $\Delta$ in Subsection \nref{subsec_affine}, so we distinguish the three cases
  for $a_0k$. Note from the Dynkin diagrams, that roots in $\bar{\Delta}$ have
  the same length as in $\Delta$ (no rescaling) except for $A_2^{(2)},
  A_{2n}^{(2)}$, which is the case $a_k=4$.\\

  Let $a_0k=1$, i.e. $\Delta$ an untwisted affine Lie algebra. Then
  $\Delta^{re}=\bar{\Delta}+\delta\Z$ and we only have to check $t=4$.
  Obviously $\alpha+\delta\Z$  with $\alpha\in \bar{\Delta}$ is long
  (i.e. length $4$ resp $6$) iff  $\alpha$ is. We hence finds
  $$\Delta^{t,re}=\bar{\Delta}^{long}+\delta\Z$$
  We want to prove that $\Delta^{t}$ is indeed the root
  system of the
  untwisted affine Lie algebra associated to $\bar{\Delta}^{t}$ (note
  $D_3^{(1)}:=A_3^{(1)}$); for
  the latter we have determined the root system and a set of simple roots
  $\alpha_i^t,\;1\leq i\leq n$ in Lemma \nref{lm_longFinite}. We further notice
  from the Dynkin diagram of the untwisted affine Lie algebras that
  $\alpha_0^t:=\alpha_0$ is always a long root. We start with the cases where
  $\bar{\Delta}^t$ is connected, in these cases $\Pi^t:=\{\alpha_0^t\}\cup
  \bar{\Pi}^t$ will already be a set of simple roots for the affinization:
  \begin{itemize}
  \item Let $\Delta=B_n^{(1)},\;n\geq 3$ and $t=4$, then $\bar{\Delta}^t=D_n$
    (note $D_3=A_3$) with simple roots $\alpha_i^t:=\alpha_i,\;1\leq i\leq n-1$
    and
    $\alpha_n^t:=\alpha_{n-1}+2\alpha_n$. We convince ourselves that
    $\Pi^t:=\{\alpha_0^t\}\cup \bar{\Pi}^t$ with $\alpha_0^t:=\alpha_0$ is of
    type $D_n^{(1)}$ resp. $A_3^{(1)}$ (rescaled) and
    $\delta_{B_n^{(1)}}=\delta_{D_n^{(1)}}$
    under this correspondence:
    \begin{align*}
      (\alpha_0^t,\alpha_1^t)
      &=(\alpha_0,\alpha_1)=0\\
      (\alpha_0^t,\alpha_2^t)
      &=(\alpha_0,\alpha_2)=-2\\
      (\alpha_0^t,\alpha_3^t)
      &=\begin{cases}
	(\alpha_0,\alpha_2+2\alpha_3)=-2,&n=3\\
	(\alpha_0,\alpha_3)=0,& n\geq 4
        \end{cases}\\
      (\alpha_0^t,\alpha_{i}^t)
      &=0,\qquad i\geq 4\\
      \delta_{D_n^{(1)}}
      &=\alpha_0^t+\theta_{D_n}\\
      &=\alpha_0^t+\alpha_1^t+2\alpha_2^t+2\alpha_3^t+\cdots
      +\alpha_{n-1}^t+\alpha_{n}\\
      &=\alpha_0+\alpha_1+2\alpha_2+2\alpha_3+\cdots
      +\alpha_{n-1}+(\alpha_{n-1}+\alpha_n)\\
      &=\alpha_0+\theta_{B_n}=\delta_{B_n^{(1)}}
    \end{align*}
    \item Let $\Delta=F_4^{(1)}$ and $t=4$, then $\bar{\Delta}^t=D_4$ with
    simple roots $\alpha_i^t:=\alpha_1,\alpha_2,\alpha_2+2\alpha_3,
    \alpha_3+2\alpha_3+2\alpha_4$ and $\alpha_1^t$ the center node. We convince
    ourselves that
    $\Pi^t:=\{\alpha_0^t\}\cup \bar{\Pi}^t$ with $\alpha_0^t:=\alpha_0$ is of
    type $D_4^{(1)}$ (rescaled) and $\delta_{F_4^{(1)}}=\delta_{D_4^{(1)}}$
    under this correspondence:
    \begin{align*}
	(\alpha_0^t,\alpha_1^t)
	&=(\alpha_0,\alpha_1)=-2\\
	(\alpha_0^t,\alpha_2^t)
	&=(\alpha_0,\alpha_2)=0\\
	(\alpha_0^t,\alpha_3^t)
	&=(\alpha_0,\alpha_2+2\alpha_3)=0\\
	(\alpha_0^t,\alpha_4^t)
	&=(\alpha_0,\alpha_2+2\alpha_3+2\alpha_4)=0\\
	\delta_{D_4^{(1)}}
	&=\alpha_0^t+\theta_{D_4}\\
	&=\alpha_0^t+2\alpha_1^t+\alpha_2^t+\alpha_3^t+\alpha_4^t\\
	&=\alpha_0+2\alpha_1+\alpha_2+(\alpha_2+2\alpha_3)
	+(\alpha_2+2\alpha_3+2\alpha_4)\\
	&=\alpha_0+2\alpha_2+3\alpha_2+4\alpha_3+2\alpha_4\\
	&=\alpha_0+\theta_{F_4}=\delta_{F_4^{(1)}}
    \end{align*}
    \item Let $\Delta=G_2^{(1)}$ and $t=3,6$, then $\bar{\Delta}^t=A_2$ with
    simple roots $\alpha_i^t:=\alpha_1,\alpha_1+3\alpha_2$. We convince
    ourselves that
    $\Pi^t:=\{\alpha_0^t\}\cup \bar{\Pi}^t$ with $\alpha_0^t:=\alpha_0$ is of
    type $A_2^{(1)}$ (rescaled) and $\delta_{G_2^{(1)}}=\delta_{A_2^{(1)}}$
    under this correspondence:
    \begin{align*}
	(\alpha_0^t,\alpha_1^t)
	&=(\alpha_0,\alpha_1)=-3\\
	(\alpha_0^t,\alpha_2^t)
	&=(\alpha_0,\alpha_1+3\alpha_2)=-3\\
	\delta_{A_2^{(1)}}
	&=\alpha_0^t+\theta_{A_2}\\
	&=\alpha_0^t+\alpha_1^t+\alpha_2^t\\
	&=\alpha_0+\alpha_1+(\alpha_1+3\alpha_2)\\
	&=\alpha_0+\theta_{G_2}=\delta_{G_2^{(1)}}
    \end{align*}
    \end{itemize}
    We now turn to the case $C_n^{(1)},\;n\geq 2$ and $t=4$, where
    $\bar{\Delta}^t=A_1^{\times n}$ is disconnected. Here we wish to prove
    $\Delta^t=(A_1^{(1)})^{\times n}$. This is slightly more complicated than
    the previous cases, because we want multiple affinizations resp. the set of
    simple roots $\Pi^t:=\{\alpha_0\}\cup \bar{\Pi}^t$ will not suffice to
    generate $\Delta^t$. We proceed ad-hoc: Consider the diagram
    automorphism $f$ on $C_n^{(1)}$ switching $\alpha_i\leftrightarrow
    \alpha_{n-i}$. Define $\Pi^t:=\bar{\Pi}^t \cup f(\bar{\Pi}^t)$
    (which are long roots). We calculate the Cartan matrix of $\Pi^t$ and
    verify it is indeed $(A_1^{(1)})^{\times n}$ with each $\alpha_i^t$ a
    simple root and $\alpha_{0,i}^t:=f(\alpha_{n-i+1}^t)$ the respective
    affinization node:
    \begin{align*}
	(\alpha_i^t,\alpha_{0,j}^t)
	&=(\alpha_i^t,f(\alpha_{n-j+1}^t))
	=(\alpha_n+2\alpha_{n-1}+\cdots+2\alpha_{n-i+1},
	\alpha_0+2\alpha_{1}+\cdots+2\alpha_{n-j})=
    \end{align*}
    We distinguish three cases:
    \begin{itemize}
      \item If $i<j$ then clearly $(\alpha_i^t,\alpha_{0,j}^t)=0$
      \item If $i>j$ we have $0<n-i+1\leq n-j<n$ and consider the nonempty
      subsum $2\beta:=2\alpha_{n-i+1}+\cdots+2\alpha_j$. Since $\beta$ is in
      the in the parabolic $A_{i-j}$ subsystem of $C_n^{(1)}$ it is a short
      root. Hence:
      \begin{align*}
	(\alpha_i^t,\alpha_{0,j}^t)
	&=(\alpha_n+2\alpha_{n-1}+\cdots+2\alpha_{n-j+1}+2\beta,
	\alpha_0+2\alpha_1+\cdots+2\alpha_{n-i}+2\beta)\\
	&=2(\alpha_n+2\alpha_{n-1}+\cdots+2\alpha_{n-j+1},\beta)
	+2(\beta,\alpha_0+2\alpha_1+\cdots+2\alpha_{n-i})
	+4(\beta,\beta)\\
	&=-4-4+8=0
      \end{align*}
      \item If $i=j$ we calculate:
      \begin{align*}
	(\alpha_i^t,\alpha_{0,i}^t)
	&=(\alpha_n+2\alpha_{n-1}+\cdots+2\alpha_{n-i+1},
	\alpha_0+2\alpha_{1}+\cdots+2\alpha_{n-i})\\
	&=\begin{cases}
		(\alpha_n,2\alpha_{n-1})=-4,& i=1\\
		(2\alpha_{n-i+1},2\alpha_{n-1})=-4, & 1<i<n\\
		(2\alpha_1,\alpha_0)=-4, &i=0
	  \end{cases}
      \end{align*}
      Altogether we have shown that $\Pi^t$ defined above generates a
      (rescaled) root system of type $(A_1^{(1)})^{\times n}$. We now check
      explicitly that this already accounts for all roots in $\Delta^{t,re}=\pm
      \alpha_i^t+\delta\Z$: This is again, because in \emph{every}
      copy of $A_1^{(1)}$ we have the same $\delta$ under the
      correspondence:
      \begin{align*}
	\delta_{(A_1^{(1)})_i}
	&=\alpha_{0,i}^t+\theta_{(A_1)_i}
	=\alpha_{0,i}^t+\alpha_i^t\\
	&=\left(\alpha_n+2\alpha_{n-1}+\cdots+2\alpha_{n-i+1}\right)
	+\left(\alpha_0+2\alpha_{1}+\cdots+2\alpha_{n-i}\right)\\
	&=\alpha_0+2\alpha_1+2\alpha_2+\cdots+2\alpha_{n-1}+\alpha_n\\
	&=\alpha_0+\theta_{C_n}=\delta_{C_n^{(1)}}
      \end{align*}
    \end{itemize}
    
    Let $a_0k=2,3$, i.e.
    $\Delta=D_{n+1}^{(2)},E_6^{(2)},A_{2n-1}^{(2)},D_4^{(3)}$. Note
    that we still have $a_0=1$, but now $\theta$ is the highest short root,
    $\alpha_0$ is always a short root (hence not in $\Delta^t$) and
    the set of roots is 
    $$\Delta^{re}=\left(\bar{\Delta}^{short}+\delta\Z\right)
    \;\cup\; \left(\bar{\Delta}^{long}+\delta k\Z\right)$$
    We wish to prove uniformly that for these twisted affine Lie
    algebras $\Delta^t$ is the same as for the respective untwisted affine
    Lie algebra $\Delta'$ with same $\bar{\Delta}$. We will do so by choosing a
    long root $\alpha_0'$ such that the subsystem generated by
    $\{\alpha_0'\}\cup \Pi$ is precisely $\Delta'$ and
    $k\delta_\Delta=\delta_{\Delta'}$ under this correspondence. This then
    reduces the problem completely to $\Delta'$, since the long roots of
    $\Delta$ are precisely $\bar{\Delta}^{long}+\delta k\Z$:
    \begin{itemize}
     \item Let $\Delta=D_{n+1}^{(2)}$ and $t=4$, then we choose
     $\alpha_0':=2\alpha_0+\alpha_1$. It is clear (from the parabolic
    subsystem $B_n$ excluding $\alpha_n$) that $\alpha_0'$ is
    long and that $\{\alpha_0'\}\cup\Pi$ is of type $\Delta'=B_n^{(1)}$ and we
    check 
    \begin{align*}
	\delta_{B_n^{(1)}}
	&=\alpha_0'+\theta_{B_n}\\
	&=(2\alpha_0+\alpha_1)+(\alpha_1+2\alpha_2+\cdots 2\alpha_n)\\
	&=2\left(\alpha_0+(\alpha_1+\alpha_2+\cdots+\alpha_n)\right)\\
	&=2\left(\alpha_0+\theta^{short}_{B_n}\right)
	=2\delta_{D_{n+1}^{(2)}}
    \end{align*}
    \item Let $\Delta=A_{2n-1}^{(2)},n\geq 3$ and $t=4$, then we choose
    $\alpha_0':=2\alpha_0+2\alpha_2+2\alpha_3+\cdots+2\alpha_{n-1}+\alpha_n$
    (the highest root in the parabolic subsystem $C_n$ excluding $\alpha_1$),
    then we have to check $\{\alpha_0'\}\cup\Pi$ is indeed
    $\Delta'=C_n^{(1)}$ and $\delta_{C_n^{(1)}}=2\delta_{A_{2n-1}^{(2)}}$: It
    is clear from the subsystem that $\alpha'_0$ is orthogonal on all but
    $\alpha_1$ and we calculate that now:
    \begin{align*}
      (\alpha_0',\alpha_1)
      &=(2\alpha_0+2\alpha_2+\cdots,\alpha_1)=-2\\
      \delta_{C_n^{(1)}}
      &=\alpha_0'+\theta_{C_n}\\
      &=(2\alpha_0+2\alpha_2+2\alpha_3+\cdots+2\alpha_{n-1}+\alpha_n)\\
      &+(2\alpha_1+2\alpha_2+\cdots+2\alpha_{n-1}+\alpha_n)\\
    &=2\left(\alpha_0+\alpha_1+2\alpha_2+\cdots+2\alpha_{n-1}+\alpha_n\right)\\
      &=2\left(\alpha_0+\theta^{short}_{C_n}\right)
      =2\delta_{A_{2n-1}^{(2)}}
    \end{align*}
    \item Let $\Delta=E_6^{(2)}$ and $t=4$, then we choose
    $\alpha_0':=2\alpha_0+2\alpha_1+2\alpha_2+\alpha_3$ (the highest
    \emph{long} root of the parabolic subsystem $C_4$), then we have to check  
    $\{\alpha_0'\}\cup\Pi$ is indeed $\Delta'=F_4^{(1)}$ (note that $\alpha_0$
    is at the other end of the diagram!) and
    $\delta_{F_4^{(1)}}=2\delta_{E_6^{(2)}}$:
    \begin{align*}
      (\alpha_0',\alpha_1)
      &=(2\alpha_0+2\alpha_1+2\alpha_2+\alpha_3,\alpha_1)
      =-2+4-2=0\\
      (\alpha_0',\alpha_2)
      &=(2\alpha_0+2\alpha_1+2\alpha_2+\alpha_3,\alpha_2)
      =-2+4-2=0\\
      (\alpha_0',\alpha_3)
      &=(2\alpha_0+2\alpha_1+2\alpha_2+\alpha_3,\alpha_3)
      =-4+4=0\\
      (\alpha_0',\alpha_4)
      &=(2\alpha_0+2\alpha_1+2\alpha_2+\alpha_3,\alpha_4)
      =-2\\
      \delta_{F_4^{(1)}}
      &=\alpha'+\theta_{F_4}\\
      &=(2\alpha_0+2\alpha_1+2\alpha_2+\alpha_3)
      +(2\alpha_1+4\alpha_2+3\alpha_3+2\alpha_4)\\
      &=2\left(\alpha_0+2\alpha_1+3\alpha_2+2\alpha_3+\alpha_4\right)\\
      &=2\left(\alpha_0+\theta^{short}_{F_4}\right)
      =2\delta_{E_6^{(2)}}\\
    \end{align*}
    \item Let $\Delta=D_4^{(3)}$ and $t=3,6$, then we choose
    $\alpha_0':=3\alpha_0+3\alpha_1+\alpha_2$ (a long root in the $G_2$
    subsystem generated by $\alpha_0+\alpha_1,\alpha_2$), then we have to check
    $\{\alpha_0'\}\cup\Pi$ is indeed $\Delta'=G_2^{(1)}$ (note that $\alpha_0$
    is at the other end of the diagram!) and
    $\delta_{G_2^{(1)}}=3\delta_{D_4^{(3)}}$:
    \begin{align*}
      (\alpha_0',\alpha_1)
      &=(3\alpha_0+3\alpha_1+\alpha_2,\alpha_1)
      =-3+6-3=0\\
      (\alpha_0',\alpha_2)
      &=(3\alpha_0+3\alpha_1+\alpha_2,\alpha_2)
      =-9+6=-3\\
      \delta_{G_2^{(1)}}
      &=\alpha'+\theta_{G_2}\\
      &=(3\alpha_0+3\alpha_1+2\alpha_2)
      +(3\alpha_1+\alpha_2)\\
      &=3\left(\alpha_0+2\alpha_1+\alpha_2\right)\\
      &=3\left(\alpha_0+\theta^{short}_{G_2}\right)
      =3\delta_{D_4^{(3)}}\\
    \end{align*}
    \end{itemize}
    
    Let finally $a_0k=4$, then we have the more exceptional cases
    $A_2^{(2)},A_{2n}^{(2)}$ with $a_0=2,k=2$ and $\theta$ again the highest
    root and where the finite root system $\bar{\Delta}=A_1,C_n$ is rescaled
    by $2$ resp. $\sqrt{2}$. The explicit set of all roots is in both cases
    $$\Delta^{re}=\underbrace{\left(\bar{\Delta}^{short}+\delta\Z\right)}_
      {(\alpha,\alpha)=4}
    \;\cup\; \underbrace{\left(\bar{\Delta}^{long}+\delta k\Z\right)}_
      {(\alpha,\alpha)=8}  	  
    \;\cup\; \underbrace{\left(\frac{1}{2}(\bar{\Delta}^{long}+\delta)
      +\delta\Z\right)}_{(\alpha,\alpha)=2}$$
    
    Let first $A_2^{(2)}$ and $t=8$ (or equivalently $t=4$ since all roots have
    length $2,8$), then the set of roots of length $8$ is explicitly
    $$\Delta^{t,re}=\bar{\Delta}+\delta k\Z=\{\pm
      \alpha_1\}+(4\alpha_0+2\alpha_1)\Z$$
    We choose the long roots $\alpha_1^t:=\alpha_1$ and
    $\alpha_0^t:=4\alpha_0+\alpha_1=-\alpha_1+2\delta$ (the reflection of
    $\alpha_1$ on $\alpha_0$). We check that $\Pi^t=\{\alpha_0^t,\alpha_1^t\}$
    generates a subsystem of type $A_1^{(1)}$
    $$(\alpha_0^t,\alpha_1^t)=(2\delta-\alpha_1,\alpha_1)=-8$$
    and since under this correspondence
    $\delta_{A_1^{(1)}}=\alpha_0^t+\alpha_1^t=2\delta_{A_2^{(2)}}$
    we have in fact equality $\Delta^{t,re}=A_1^{(1)}$.\\
    
    Let now $A_{2n}^{(2)},\;n\geq 2$ and $t=4$ with $\bar{\Delta}=C_n$, then the
    set of roots of length $8$ is explicitly 
    $$\Delta^{re}=\left(C_n^{short}+\delta\Z\right)
    \;\cup\; \left(C_n^{long}+2\delta \Z\right)$$
    This reminds strongly on $A_{2n-1}^{(2)}$ (after rescaling the generators by
    $\sqrt{2}$). Indeed, choosing $\alpha_0':=2\alpha_0+\alpha_1$ (a long root
    in the parabolic subsystem $C_2$ generated by $\alpha_0,\alpha_1$) we
    calculate the Cartan matrix and check that also $\delta$ 
    corresponds:
    \begin{align*}
      (\alpha_0',\alpha_1)
      &=(2\alpha_0+\alpha_1,\alpha_1)=-4+4=0\\
      (\alpha_0',\alpha_2)
      &=(2\alpha_0+\alpha_1,\alpha_2)=-4\\
      \delta_{A_{2n-1}^{(2)}}
      &=\alpha_0'+\theta^{short}_{C_n}\\
    &=(2\alpha_0+\alpha_1)+(\alpha_1+2\alpha_2+\cdots+2\alpha_{n-1}+\alpha_n)\\
      &=2\alpha_0+2\alpha_1+2\alpha_2+\cdots+2\alpha_{n-1}+\alpha_n\\
      &=a_0\alpha_0+\theta_{C_n}=\delta_{A_{2n}^{(2)}}.
    \end{align*}
    This shows $\Delta^{t,re}=A_{2n-1}^{(2)}$ for $t=4$. For $t=8$ we may
    hence equally look at roots of length $4$ in $A_{2n-1}^{(2)}$, which we have
    already seen is $(A_1^{(1)})^\times n$. This concludes the proof of Theorem
    \nref{thm_longAffine}.
\end{proof}

\section{Main Theorem}\label{sec_main}

\begin{lemma}[\cite{Lusz94} Lm. 35.2.2]\label{lm_LusztigGeneric} 
Assume (35.1.2) that 
\begin{enumerate}[a)]
 \item For any $i\neq j$ with $\ell_{\alpha_j}\geq 2$ we have
  $\ell_{\alpha_i}\geq -a_{ij}+1$ (with $a_{ij}$ the Cartan matrix).
 \item The root system is without odd cycles, i.e. $\g\neq A_{n}^{(1)},\;2|n$.
\end{enumerate}
Then $_Rf$ is generated by all $E_{\alpha_i},\;\ell_{\alpha_i}\geq 2$ and
$E_{\alpha_i}^{(\ell_{\alpha_i})}$.\\
\end{lemma}

Lusztig defines in \cite{Lusz94} Chp. 36 the subalgebra $_Ru$ to be generated
by all $E_{\alpha_i},\;\ell_{\alpha_i}\neq 1$ and establishes under the
restriction on $\g,\ell$ above a Frobenius homomorphism with kernel $_Ru$. Note
that in contrast in \cite{Lusz90b} Thm. 8.3 he had defined $u$ without
restrictions on $q$ as being generated by \emph{all} root vectors $E_\alpha$
with $\ell_\alpha\neq 1$, but has refrained from doing so in the affine case by
lack of root vectors. Note that frequently already in the finite case $u$ for
$\g,\ell$ violating a) is \emph{not} generated by simple root vectors.\\

The aim of this article is the following theorem, that describes in all cases
violating a) a subalgebra $u_q^\L(\g)$ that has properties similar to $u$
defined in the finite case and could serve as a kernel of a Frobenius
homomorphism without restrictions on $\ell$.

\begin{theorem}\label{thm_main}
  Let $\g$ be an affine Lie algebra and $q$ and $\ell$-th root of unity. We
  shall define a Hopf subalgebra $u_q^{\L}(\g)^+\subset
  U_q^\L(\g)^+$  with the following properties: 
  \begin{itemize}
    \item $u_q^\L(\g)^+$ consists of all degrees (roots) $\alpha$ with
    $\ell_\alpha\neq 1$. 
    \item Except for cases marked \emph{deaffinized}, $u_q^{\L}(\g)^+$
    is generated by primitives $E_{\alpha_i^{(0)}}$, spanning a braided vector
    space $M$ (in the deaffinized cases $u_q^{\L}(\g)^+$ is an infinite
    extension tower, see Section \nref{sec_A11})
    \item Lusztig's $_Ru\subset u_q^{\L}(\g)^+$ and equality only holds for
    \emph{trivial} and \emph{generic} cases.
    \item Except for cases marked \emph{deaffinized} and
    (possibly) \emph{exotic} $u_q^\L(\g)^+$ is coradically graded and hence maps
    onto the Nichols algebra $\B(M)=u_{q'}(\g^{(0)})^+$. For untwisted affine
    $\g$ we prove (and else conjecture) they are isomorphic.
    \item Except for cases marked \emph{deaffinized}, the $u_q^{\L}(\g)^+$
    fulfills $\ell_i\neq 1$ as well as Lusztig's non-degeneracy condition
    $\ell_{\alpha_i}\geq -a_{ij}+1$. Hence we have determined subalgebras on
    which Lusztig's theory in \cite{Lusz94} may be applied.
  \end{itemize}
\begin{center}\noindent
We explicitly describe the type $\g^{(0)},q'$ of $M,\B(M)$ as follows
\begin{tabular}{ll|lll}
  $\g$ & $\ell$ & $M$ & $q'$ & comment \\
  \hline\hline
  all & $\ell=1,2$ & $\{0\}$ & $q$ & trivial\\
  \hline
  $A_1^{(1)}$ & $\ell=4$ & $A_1^{2\times}$ & $q$ & {deaffinized}\\
  $B_n^{(1)},D_{n+1}^{(2)}$ & $\ell=4$ 
    & $A_1^{2n\times}$ & $q$ & short roots, deaffinized\\
  $C_n^{(1)},A_{2n-1}^{(2)}$ & $\ell=4$
    & $D_n^{(1)}$ & $q$ & short roots\\
  $F_4^{(1)},E_6^{(2)}$ & $\ell=4$
    & $D_4^{(1)}$ & $q$ & short roots\\
  $G_2^{(2)},D_4^{(3)}$ & $\ell=3,6$
    & $A_2^{(1)}$ & $q$ & short roots\\
  \hline
  $A_2^{(2)}$ & $\ell=4$
    & $A_1^{2\times}$ & $q$ & very short roots, deaffinized\\
  $A_2^{(1)}$ & $\ell=8$
    & $A_1^{2n\times}$ & $q$ & very short roots\\
  $A_{2n}^{(2)}$ & $\ell=4$
    & $A_1^{2n\times}$ & $q$ & very short roots, deaffinized\\
  $A_{2n}^{(2)}$ & $\ell=8$
    & $A_{2n-1}^{(2)}$ & $q$ &  not very long roots\\
  \hline
  $A^{(2)}_{2}$ & $3$ & $A_2^{(1)}$ & $q$ & exotic\\
  $A^{(2)}_{2}$ & $6$ & $A_2^{(1)}$ & $-q$ & exotic\\
  $A^{(2)}_{2n}$ & $3,6$ & $A^{(2)}_{2n}$ & $q$ & (pseudo-)exotic\\
  $G_2^{(1)}$ & $4$ & $A_3^{(1)}$ & $\bar{q}$ & exotic\\
  $D_4^{(3)}$ & $4$ & $D_4^{(1)}$ & $q,\bar{q},-1$ & exotic\\
\end{tabular}
\end{center}
All cases not included in the list are \emph{generic cases}
$\g^{(0)}=\g$ with $u_q^\L(\g)^+={_R}u^+$.

\end{theorem}
\begin{proof}
  The proof of this theorem will occupy the reminder of the article and
  consists of a precise description of the respective subalgebras. The proof
  proceeds as follows:\\
  
  We first determine all cases $\g,q$ where $U^\L_q(\g)$ fails Lusztig's
  condition. This is easily done in the following Lemma \nref{lm_cases}. We
  define two subcases:
  \begin{itemize} 
    \item {\bf Degenerate Cases:} When there are simple roots with
    $\ell_{\alpha_i}=1$, then these simple root vectors $E_{\alpha_i}$ are
    not contained in $_Ru$. We give an explicit set of primitive    elements
    $E_{\alpha_i^{(0)}}$ with $\ell_{\alpha_i}\neq 1$, determine their   
    root system $\g^{(0)}$ and show it contains precisely the real roots   
    $\alpha$ of $\g$ with $\ell_\alpha\neq 1$. This is done in Section   
    \nref{sec_degenerate} by linking $\g^{(0)}$ to subsystems of the dual root
    system $\g^\vee$ of $\g$.
    \item{\bf Exotic Cases:} Now assume all $\ell_{\alpha_i}\neq 1$ but the
    condition 	
    $\ell_{\alpha_i}\geq -a_{ij}+1$ fails. In these cases all
    $E_{\alpha_{i}^{(0)}}:=E_{\alpha_i}\in {_R}u$, but usually they generate a
    Nichols algebra of different type than $\g$ and do not include all roots
    with $\ell_\alpha\neq 1$. We determine case-by-case additional primitive
    elements $E_{\alpha_{n+1}^{(0)}}$ in $U^\L_q(\g)^+$, calculate the Nichols
    algebra generated by them and verify that it contains all roots with
    $\ell_\alpha\neq 1$. This is done in Section \nref{sec_exotic}.
    \item{\bf Deaffinized Cases:} A specific case with $\ell_{\alpha_i}\neq 1$
    but violating $\ell_{\alpha_i}\geq -a_{ij}+1$ is $\g=A_1^{(1)},\ell=4$. It
    is the such only exotic case where ${_R}u$ is finite-dimensional and an
    infinite tower of copies is needed to cover all roots with 
    $\ell_\alpha\neq 1$. This case is dealt with in Section
    \nref{sec_A11}\\
    Moreover there are quite a few degenerate cases where $\g^{(0)}$ still
    fails $\ell_{\alpha_i}\geq -a_{ij}+1$, namely precisely those with
    $(\g^{(0)}=A_1^{(1)})^{\times n}$. 
  \end{itemize}
  Having determined a suitable set of primitive elements
  $E_{\alpha_i^{(0)}}$ the remaining assertions of the theorem follow by
  standard arguments:
  \begin{itemize}
   \item For degenerate cases we prove now that $u_q^\L(\g)$ is coradically
    graded: In most cases $\g^{(0)}$ has the same rank as
    $\g$ i.e. is generated by $E_{\alpha_0^{(0)}},\ldots E_{\alpha_n^{(0)}}$
    with $\alpha_i^{(0)}$ a basis of $\mathbb{R}^\Pi$. Then there is a linear
    function $f:\N^\Pi\rightarrow \mathbb{R}$ fulfilling $f(\alpha_i^{(0)})=1$
    and since $U_q^\L(\g)$ is $\N^\Pi$-graded, the assertion follows. For the
    exceptional cases with $\g^{(0)}=(A_1^{(1)})^{\times n}$, i.e. rank of
    $\rank(\g^{(0)})=2n=2\rank(\g)-2$ we convince ourselves from the specific 
    $\alpha_i^{0}$ given in Lemma \nref{lm_primitive}:
    $$\alpha_1^{(0)},\ldots \alpha_n^{(0)}=\alpha_n+\alpha_{n-1}+\ldots$$
    $$\alpha_{n+1}^{(0)},\ldots \alpha_{2n}^{(0)}=\alpha_0+\alpha_{1}+\ldots$$
    that nevertheless the linear function $f(\alpha_0)=f(\alpha_n)=1$ and
    $f(\alpha_i)=0$ else fulfills $f(\alpha_i^{(0)})=1$. This again shows the
    assumption. 
    \item By the universal property of the Nichols algebra, every Hopf algebra
    generated by a braided vector space $M$ of primitive elements maps onto the
    Nichols algebra $\B(M)$. For the exotic cases the weaker statement is that
    $\gr(u_q^\L(\g))$ maps onto $\B(M)$.
    \item By construction the root system of $\B(M)$ precisely coincides with
    the roots contained in $u_q^\L(\g)$. For untwisted Lie algebras $\g$ we
    have a PBW-basis of $U_q^\L(\g)$ by \cite{Beck94} Prop. 6.1, which implies
    the map to $\B(M)$ is an isomorphism. We conjecture this to be true also
    for the twisted affine $\g$. \\
    For the deaffinized cases it is clear we do not
    get an isomorphism to the (finite-dimensional) Nichols algebra, but we
    nevertheless expect an isomorphism to the explicit extension tower algebra
    $u_q((A_1^{(n)})^{\times n})$. 
  \end{itemize}
\end{proof}

We easily determine all cases in question, note that we are in the following
only interested in cases violating Lusztig's first condition:

\begin{lemma}\label{lm_cases}
  The cases of affine quantum groups $\g,\ell$ where the set
  $E_{\alpha_i^{(0)}}$ fails Lusztig's condition are the following:
  \begin{enumerate}[a)]
    \item Cases with $\Delta^{\vee,t}=\Delta^\vee$ i.e.
    $\alpha_i^{(0)}=\alpha_i$ which have \emph{not} yet
    appeared as being exceptional, but violate either condition of Lemma
    \nref{lm_LusztigGeneric}:
    \begin{center}
    \begin{tabular}{ll}
      $\g,\Delta$ & $\ell$ \\
      \hline\hline
      $A_1^{(1)}$ & $4$ \\
      $A^{(2)}_{2},A^{(2)}_{2n}$ & $3,6$\\
      $G_2^{(1)},D_4^{(3)}$ & $4$\\
      \hline
      $A_n^{(1)},2|n\quad $ & $\neq 1,2$ \\
    \end{tabular}
    \end{center}
    Note that also the only finite exotic example \nref{exm_G2} $\g=G_2,\ell=4$
    would fall into this case.
    \item Cases with $\varnothing\neq \Delta^{\vee,t}\neq\Delta^\vee$ where the
    root subsystem $\left(\Delta^{\vee,t}\right)^\vee$ is in case a) and
    hence violates either condition in Lemma \nref{lm_LusztigGeneric}: 
    \begin{center}
    \begin{tabular}{ll|l}
    $\g,\Delta$ & $\l$ &  $\left( (\Delta^\vee)^t\right)^\vee$ \\
    \hline\hline
    $B_n^{(1)},D_{n+1}^{(2)},A_2^{(2)},A_{2n}^{(2)}$
    & $4$ & $\left(A_1^{(1)}\right)^{\times n}$\\
    \hline
    $G_2^{(1)},D_4^{(3)}$ & $3,6$ & $A_2^{(1)}$ 
  \end{tabular}
  \end{center}
  \end{enumerate}
\end{lemma}
\begin{proof}
\begin{enumerate}[a)]
 \item We assume $\Delta^{\vee,t}=\Delta^\vee$. Besides clarifying the root
  system in question, this assumption will exclude several values for $\ell$,
for which we  have to consider the appropriate subsystem in b). For the first
condition in  Lemma \nref{lm_LusztigGeneric} we reformulate
  \begin{align*}
    \ell_{\alpha_i}
    &> -a_{ij}\\
    \Leftrightarrow\quad
    \frac{\ell}{\gcd(\ell,(\alpha_i,\alpha_i))}
    &> -\frac{2(\alpha_i,\alpha_j)}{(\alpha_i,\alpha_i)}\\
    \Leftrightarrow\quad
    \lcm(\ell,(\alpha_i,\alpha_i))
    &>-2(\alpha_i,\alpha_j)
  \end{align*}
  Hence we have to check all cases with $\lcm\leq
  -2(\alpha_i,\alpha_j)=0,2,4,6,8$. For $\ell=1,2$ we have
  $\ell\mid(\alpha,\alpha)$ hence $\Delta^{\vee,t}=\varnothing$ and this is not
  included in this case, also $(\alpha_i,\alpha_i)=2,4,6,8$, this leaves
  $$\ell=3,6,\;(\alpha_i,\alpha_i)=2,6,\;-2(\alpha_i,\alpha_j)=6,8$$
  $$\ell=4,\;(\alpha_i,\alpha_i)=2,4,\;-2(\alpha_i,\alpha_j)=4,6,8$$
  Notice now that case-by-case several of these cases are degenerate in the
  sense that $q^{(\alpha_k,\alpha_k)}=1$ and hence do not belong in this case
  a): For  $\ell=3,6$ cases with  $(\alpha_i,\alpha_i)=6$ (namely
  $G_2^{(1)},D_4^{(3)}$)  are degenerate, which leaves only
  $A^{(2)}_{2},A^{(2)}_{2n}$. For $\ell=4$  cases with $(\alpha_i,\alpha_i)=4,8$
  are degenerate while simply-laced $\g$  have $-2(\alpha_i,\alpha_j)=-2$ ,
  which only leaves the cases $A_1^{(1)},G_2^{(1)},D_4^{(3)}$.  \\
  The second condition in Lemma \nref{lm_LusztigGeneric} is violated for
  $A_n^{(1)}$ for $2|n$ and arbitrary $\ell\neq 1,2$ (again for $\ell=1,2$ it
  is degenerate).
  \item If $\Delta^{t,\vee}\neq \Delta^\vee$ violates the conditions in Lemma
  \nref{lm_LusztigGeneric} then it has to appear in a) above. We simply check
  the table in Lemma \nref{lm_primitive} and find that the only cases are for
one
  $B_n^{(1)},D_{n+1}^{(2)},A_2^{(2)},A_{2n}^{(2)}$ at $\ell=4$ which yield
  $(A_1^{(1)})^{\times n}$. On the other hand we have $G_2^{(1)},D_4^{(3)}$ for
  $\ell=3,6$ which yield $A_2^{(1)}$ having an odd cycle.
\end{enumerate}
\end{proof}

\section{Degenerate cases}\label{sec_degenerate}

\subsection{Preliminaries on primitives}
We start by the following observation in our context, that appears e.g.
throughout \cite{Heck09}. We shall use it in what follows to construct
exceptionally primitive elements for small $\ell$. Note that Lusztig's
reflection operator is defined in terms of the Cartan matrix of $\g$, not
intrinsically with respect to the braiding matrix (which may be of different
type) in \cite{Heck09} and also may not always be expressed as iterated braided
commutators.

\begin{lemma}\label{lm_reflectionPrimitive}
    Let $\alpha,\beta$ such that $\ell | 2(\alpha,\beta)$, or
    equivalently $q^{(\alpha,\beta)}=\pm 1$ i.e. the braiding is
    symmetric. Then 
    \begin{enumerate}[a)]
      \item The braided commutator $[x,y]$ of primitive elements in
      degree $\alpha,\beta$ is again primitive.
      \item Let $\alpha=\alpha_i$ and $x$ a primitive element in degree $\beta$
      with braided commutator $[E_{\alpha_i},x]=[F_{\alpha_i},x]=0$, then
      Lusztig's reflection $T_{i,1}''(x)$ is again primitive.
      \item Let $\ell|(\alpha_i,\alpha_i)$, i.e. trivial
      self-braiding $q^{(\alpha_i,\alpha_i)}=1$, then $T_{i,1}''$ maps primitive
      elements to primitive elements without further assumptions.
    \end{enumerate}
\end{lemma}
\begin{proof}
\begin{enumerate}[a)]
    \item This is a standard argument: Let the braiding of some elements
    $x_1,x_2$ be given by $x_1\otimes x_2\mapsto q_{12}x_2\otimes x_1$ and  
    $x_2\otimes x_1\mapsto q_{21}x_1\otimes x_2$. Then the assumption
    $q_{12}q_{21}=1$ implies for primitive elements $x_1,x_2$:
    \begin{align*}
      \Delta([x_1,x_2])
      &:=\Delta(x_1x_2-q_{12}x_2x_1)\\
      &=x_1x_2\otimes 1+x_1\otimes x_2+q_{12}x_2\otimes x_1+1\otimes x_1x_2\\
      &-q_{12}\left(x_2x_1\otimes 1+x_2\otimes x_1
      +q_{21}x_1\otimes x_2+1\otimes x_2x_1\right)\\
      &=[x_1,x_2]\otimes 1+1\otimes [x_1,x_2]
    \end{align*}
    Note this is the only case where $[x_1,x_2]=\pm [x_2,x_2]$ are linearly
    dependent. 
    \item In \cite{Lusz94} Sec. 37.3 a relation between reflection and
    comultiplication in $U_q^\L$ is given as follows:
    \begin{align*}
      (T_{i,-1}'\otimes T_{i,-1}')\Delta(T_{i,1}''x)
      &=\left(\sum_{n}q_i^{n(n-1)/2}\{n\}_iF_i^{(n)}\otimes E_i^{(n)}\right)
	\Delta(x)\\
      &\cdot \left(\sum_{n}(-1)^nq_i^{-n(n-1)/2}\{n\}_iF_i^{(n)}\otimes
	E_i^{(n)}\right)\\
      \mbox{where }\{n\}_{\alpha_i}
      &=\prod_{a=1}^n\left(q_i^{a}-q_i^{-a}\right)
      =(q_i-q_i^{-1})^n\cdot [n]_{q_i}!=(q_i-q_i^{-1})^n\cdot [n]_{q^{-1}_i}!\\
      T_{i,-1}'T_{i,1}''
      &=\id
    \end{align*}
    With the assumption $E_ix-q^{n(\alpha_i,\beta)}xE_i=0$ and
    $F_ix-q^{-n(\alpha_i,\beta)}xF_i=0$ we calculate
    \begin{align*}
    {n}_{q_i}E_i^{(n)}x{n'}_{q_i}E_i^{(n')}
    &=(q_i-q_i^{-1})^{n+n'}E_i^nxE_i^{n'}
    =(q_i-q_i^{-1})^{n+n'}q^{n(\alpha_i,\beta)}\cdot xE_i^{n+n'}\\
    {n}_{q_i}F_i^{(n)}x{n'}_{q_i}F_i^{(n')}
    &=(q_i-q_i^{-1})^{n+n'}F_i^nxF_i^{n'}
    =(q_i-q_i^{-1})^{n+n'}q^{-n(\alpha_i,\beta)}\cdot xF_i^{n+n'}
    \end{align*}
    With $\Delta(x)=x\otimes 1+1\otimes x$ we calculate
    \begin{align*}
      &(T_{i,-1}'\otimes T_{i,-1}')\Delta(T_{i,1}''x)\\
      &=\left(\sum_{n}q_i^{n(n-1)/2}\{n\}_iF_i^{(n)}\otimes E_i^{(n)}\right)
	(x\otimes 1+1\otimes x)\\
      &\cdot \left(\sum_{n'}(-1)^{n'}q_i^{-n'(n'-1)/2}\{n'\}_iF_i^{(n')}\otimes
	E_i^{(n')}\right)\\
      &=\sum_{n,n'=0}^{\ell_i}
	q_i^{n(n-1)/2-n'(n'-1)/2}(-1)^{n'}q^{n(\alpha_i,\beta)}
	\cdot F_i^{(n)}xF_i^{(n')}\otimes E_i^{(n)}E_i^{(n')}\\
      &+\sum_{n,n'}
	q_i^{n(n-1)/2-n'(n'-1)/2}(-1)^{n'}q^{n(\alpha_i,\beta)}
	\cdot F_i^{(n)}F_i^{(n')}\otimes E_i^{(n)}xE_i^{(n')}\\
      &=\sum_{n,n'}
	q_i^{n(n-1)/2-n'(n'-1)/2}(-1)^{n'}(q_i-q_i^{-1})^{n+n'}
	q^{2n(\alpha_i,\beta)}
	\cdot xF_i^{n+n'}\otimes E_i^{(n)}E_i^{(n')}\\
      &+\sum_{n,n'}
	q_i^{n(n-1)/2-n'(n'-1)/2}(-1)^{n'}(q_i-q_i^{-1})^{n+n'}
	\cdot F_i^{(n)}F_i^{(n')}\otimes x E_i^{n+n'}x\\
      &=\sum_{m=0}q_i^{-m^2}(q_i-q_i^{-1})^m
	\left(\sum_{n=0}^m(-1)^n\begin{bmatrix}m\\n\end{bmatrix}_{q_i}
	q^{2n(\alpha_i,\beta)}q^{n(m-1)}\right)
	\cdot xF_i^{m}\otimes E_i^{m}\\
      &+\sum_{m=0}q_i^{-m^2}(q_i-q_i^{-1})^m(-1)^m
	\left(\sum_{n=0}^m(-1)^n \begin{bmatrix}m\\n\end{bmatrix}_{q_i}
	q^{n(m-1)}\right)
	\cdot F_i^{m}\otimes xE_i^{m}
\end{align*}
where we substituted $m=n+n'$. For $m>0$ the second sum
vanishes by \cite{Lusz94} Sec. 1.3.4:
$$\sum_{n=0}^m(-1)^n\begin{bmatrix}m\\n\end{bmatrix}_{q_i}q^{n(m-1)}=0$$
The first sum vanishes for $m>0$ the same reason because of the additional
assumption $q^{2(\alpha_i,\beta)}=1$. Hence only the term $m=0$ remains and by
the inverse property $T_{i,-1}'T_{i,1}''=\id$ shows the assertion
\begin{align*}
      (T_{i,-1}'\otimes T_{i,-1}')\Delta(T_{i,1}''(x))
      &=x\otimes 1+1\otimes x\\
      \Delta(T_{i,1}''x)
      &=T_{i,1}''(x)\otimes 1+1\otimes T_{i,1}''(x)
\end{align*}
\item The assumption $q^{(\alpha_i,\alpha_i)}=1$ amounts to
$q_i=q^{(\alpha_i,\alpha_i)/2}=\pm 1$ hence $\{n\}_{\alpha_i}=0$ except
$\{0\}_{\alpha_i}=1$. Hence only the terms $n,n'=0$ in b) remain which shows
the assertion without using the assumption on the vanishing braided commutator.
\end{enumerate}
\end{proof}

\subsection{Primitives from the dual root system}
Let $\g$ be an affine Lie algebra with root system $\Delta$ of rank $n+1$ and
consider the dual root system $\Delta^\vee$. In Theorem \nref{thm_longAffine} we
have determined the type and a set $\Pi^t$ of simple roots
$\alpha_0^t,\ldots,\alpha_N^t$ (or empty) for the subsystem $(\Delta^\vee)^t$ of
roots with length divisible by $t$.\\
Suppose $q$ an $\ell$-th root of unity and let $s\in\{2,4,6,8\}$ denote the
longest root length in $\g$. We find below for each $\g,\ell$ a (not unique)
value $t$, such that $\ell \nmid a$ is equivalent to $t\mid \frac{2s}{a}$ for
all root lengths $a\neq 0$ (mostly $t=\ell$). This will allow us to characterize
roots with $q^{(\alpha,\alpha)}\neq 1$, i.e. $\ell\nmid (\alpha,\alpha)$, as
being dual to roots $\alpha^\vee\in (\Delta^\vee)^t$, i.e. $t\mid
(\alpha^\vee,\alpha^\vee)$.\\
More precisely we ultimately wish to prove the following:

\begin{lemma}\label{lm_primitive}
  Via reflection we define the following root vectors
  $E_{\alpha_i^{(0)}}$ for short roots $\alpha_i^{(0)}\in
  \left((\Pi^\vee)^{t}\right)^\vee$. These are primitive elements in
  $u_q^\L(\g)^+$ resp. skew-primitives in $u_q^\L(\g)$:
  \begin{center}
  \begin{tabular}{ll|lll|ll}
    $\g,\Delta$ & $\l$ & $\Delta^\vee$ & $t$ & $(\Delta^\vee)^t$ &
    $\left((\Delta^\vee)^t\right)^\vee$ & $\left((\Pi^\vee)^{t}\right)^\vee$ \\
    \hline
    \hline
    $B_n^{(1)}$ & $4$ & $A_{2n-1}^{(2)}$ & $4$ & $\left(A_1^{(1)}\right)^{\times
      n}$ & $\left(A_1^{(1)}\right)^{\times n}$ & $\alpha_n,
      \alpha_n+\alpha_{n-1}, \alpha_n+\alpha_{n-1}+\alpha_{n-2},\ldots$\\
    &&&&&& $\alpha_0',\alpha_0'+\alpha_{1},    
      \alpha_0'+\alpha_{1}+\alpha_{2},\ldots$\\
    &&&&&&
    $\quad\alpha_0':=\alpha_0+\alpha_2+\alpha_3+\cdots+\alpha_{n-1}+\alpha_n$\\
    $C_n^{(1)}$ & $4$ &  $D_{n+1}^{(2)}$ & $4$ & $D_n^{(1)}$ & $D_n^{(1)}$ &
      $\alpha_0',\alpha_1,\alpha_2,\cdots,\alpha_{n-1},\alpha_{n-1}+\alpha_n$\\
    &&&&&& $\quad\alpha_0':=\alpha_0+\alpha_1$\\
    $F_4^{(1)}$ & $4$ & $E_6^{(2)}$ & $4$ & $D_4^{(1)}$ & $D_4^{(1)}$ &
      $\alpha_0',\alpha_4,\alpha_3,\alpha_3+\alpha_2,
      \alpha_3+\alpha_2+\alpha_1$\\
    &&&&&& $\quad\alpha_0':=\alpha_0+\alpha_1+\alpha_2+\alpha_3$\\
    $G_2^{(1)}$ & $3,6$ & $D_4^{(3)}$ & $3,6$ & $A_2^{(1)}$ & $A_2^{(1)}$ &
    $\alpha_0',\alpha_2,\alpha_2+\alpha_1$\\
    &&&&&&$\quad\alpha_0':=\alpha_0+\alpha_1+\alpha_2$\\
    \hline
    $D_{n+1}^{(2)}$ & $4$ &  $C_n^{(1)}$ & $4$ &
      $\left(A_1^{(1)}\right)^{\times n}$ & $\left(A_1^{(1)}\right)^{\times n}$&
      $\alpha_n,\alpha_n+\alpha_{n-1},
      \alpha_n+\alpha_{n-1}+\alpha_{n-2},\cdots$\\
    &&&&&& $\alpha_0,\alpha_0+\alpha_{1},
      \alpha_0+\alpha_{1}+\alpha_{2},\ldots$\\
    $A_{2n-1}^{(2)}$ & $4$ & $B_n^{(1)}$ & $4$ & $D_n^{(1)}$  & $D_n^{(1)}$ &
      $\alpha_0,\alpha_1,\alpha_2,\cdots,\alpha_{n-1},\alpha_{n-1}+\alpha_n$\\
    $E_6^{(2)}$ & $4$ & $F_4^{(1)}$ & $4$ & $D_4^{(1)}$ & $D_4^{(1)}$ &
      $\alpha_0,\alpha_1,\alpha_2,\alpha_2+\alpha_3, 
      \alpha_2+\alpha_3+\alpha_4$\\
    $D_4^{(3)}$ & $3,6$ & $G_2^{(1)}$ & $3,6$ & $A_2^{(1)}$ & $A_2^{(1)}$ &
      $\alpha_0,\alpha_1,\alpha_1+\alpha_2$ \\
    \hline
    $A_{2}^{(2)}$ & $4,8$ & $A_{2}^{(2)}$ & 
      $4,8$ & $A_1^{(1)}$ & $A_1^{(1)}$ &
      $\alpha_1+\alpha_0,\alpha_0$ \\
    $A_{2n}^{(2)}$ & $4$ & $A_{2n}^{(2)}$ & $8$ & 
      $\left(A_1^{(1)}\right)^{\times n}$ & $\left(A_1^{(1)}\right)^{\times n}$ 
      & $\alpha_0,\alpha_0+\alpha_{1},   
      \alpha_0+\alpha_{1}+\alpha_{2},\ldots$\\
    &&&&&& $\alpha_n'',\alpha_n''+\alpha_{n-1},
      \alpha_n''+\alpha_{n-1}+\alpha_{n-2},\ldots$\\
  &&&&&&$\quad\alpha_n'':=\alpha_n'+\alpha_{n-2}+\alpha_{n-3}+\cdots+\alpha_{1}
      +\alpha_0$\\  
    &&&&&&$\quad \alpha_n':=\alpha_n+\alpha_{n-1}$\\
    $A_{2n}^{(2)}$ & $8$ & $A_{2n}^{(2)}$ & $4$ & $A_{2n-1}^{(2)}$ &
      $B_n^{(1)}$ &
      $\alpha_n',\alpha_{n-1},\ldots,\alpha_0$\\
    &&&&&&$\quad \alpha_n':=\alpha_n+\alpha_{n-1}$
  \end{tabular}
  \end{center}
  Note in the last block we have a nontrivial self-duality
  $\alpha_k\leftrightarrow \alpha_{n-k}$ as well as $t\neq \ell$.
\end{lemma}~
\begin{remark}~
\begin{itemize}
  \item For finite root systems the lemma has been proven
  case-by-case as part of the proof of Thm 5.4 in \cite{Len14c}. The proof
  strategy via reflections is new.
  \item There is no coherent definition of affine root vectors available. Here
we  a-priori take some reflection to yield some $E_\alpha$, but it  will become 
clear in the proof that up to a sign our specific $E_\alpha$ are  independent 
of the choice of the reflection.
  \item We do not claim these are all primitives. This will depend on the
  Nichols algebra structure and fail precisely in the so-called exotic cases.
  \item The theorem gives only information about real roots. Imaginary roots
  are by definition included in any $(\Delta^\vee)^t$, but never fulfill
  $q^{(\delta,\delta)}\neq 1$. Whether they are included in $u_q(\g)$ will
  depend on the Nichols algebra structure and will be unexpected for the
  exotic cases.
\end{itemize}
\end{remark}
\begin{proof}
First we check that the conditions on the roots $\alpha^\vee\in(\Delta^\vee)^t$
defined in Lemma \nref{lm_DeltaT} for the dual root system $\Delta^\vee$ matches
precisely $q^{(\alpha,\alpha)}\neq 1$ i.e. $E_\alpha$ should be in $u_q^\L$ for
$q$ a primitive $\ell$-th root of unity:

\begin{lemma}\label{lm_lt}
  Let $s\in\{2,4,6,8\}$ denote the longest root length in $\g$. Consider the
  map 
  \begin{align*}
  f:\Delta^\vee
  &\to \Delta\\
  \alpha'
  &\mapsto \sqrt{\frac{s}{2}}(\alpha')^\vee,
  \qquad (\alpha')^\vee=\frac{(\alpha',\alpha')}{2}\alpha'
  \end{align*}
  Then $f$ maps $(\Delta^\vee)^t$  bijectively to the set of roots
  $\alpha\in\Delta$ fulfilling $q^{(\alpha,\alpha)}\neq 1$ for the pairs
  $\ell,t$ in Lemma \nref{lm_primitive}. Moreover, for $\ell=1,2$ no roots
  fulfill  $q^{(\alpha,\alpha)}\neq 1$, hence $t=\infty$ would be appropriate
  and for other $\ell$ all roots fulfill  $q^{(\alpha,\alpha)}\neq 1$,
  hence $t=1$ would be appropriate.
\end{lemma}
\begin{proof}
  By definition the condition $q^{(\alpha,\alpha)}=1$ amounts to
  $\ell\mid(\alpha,\alpha)$. Let $(\alpha',\alpha')=:a$, then we calculate
  \begin{align*}
      (f(\alpha'),f(\alpha'))
      &=\frac{s}{2}\cdot \left(\frac{2}{(\alpha',\alpha')}\right)^2\cdot
	(\alpha',\alpha')
      =\frac{2s}{(\alpha',\alpha')}
  \end{align*}
  Hence $\ell\mid(\alpha,\alpha)$ is equivalent to
  $\ell\mid\frac{2s}{(\alpha',\alpha')}$. We check all cases:
  \begin{itemize}
   \item For $\ell=1,2$ all $\ell\mid(\alpha,\alpha)$ and we will excluding this
    trivial cases in the following.
   \item For $s=4$ and $\ell=4$ we have $\ell\mid(\alpha,\alpha)$ for long roots
    $\alpha$ and thus dually 
    $\ell\mid\frac{2s}{(\alpha',\alpha')}=\frac{8}{(\alpha',\alpha')}$
    for short roots $\alpha'$. On the other hands, short roots are characterized
    by being not divisible by $t=4$.
  \item For $s=6$ and $\ell=3,6$ we have $\ell\mid(\alpha,\alpha)$ for long
roots
    $\alpha$ and thus dually 
    $\ell\mid\frac{2s}{(\alpha',\alpha')}=\frac{12}{(\alpha',\alpha')}$
    for short roots $\alpha'$. On the other hands, short roots are characterized
    by being not divisible by $t=3,6$.
  \item For $s=8$, i.e. $(\alpha,\alpha)=2,4,8$ we have two cases $\ell=4$ and
    $\ell=8$. For $\ell=4$ we have $\ell\mid(\alpha,\alpha)$ for all but very
    short roots $\alpha$ and thus dually 
    $\ell\mid\frac{2s}{(\alpha',\alpha')}=\frac{16}{(\alpha',\alpha')}$
    for all but very long roots $\alpha'$. On the other hands, such roots are
    characterized by being not divisible by $t=8$. For $\ell=8$ we have
    $\ell\mid(\alpha,\alpha)$ for very long roots $\alpha$ and thus dually 
    $\ell\mid\frac{2s}{(\alpha',\alpha')}=\frac{16}{(\alpha',\alpha')}$
    for very short roots $\alpha'$. On the other hands, such roots are
    characterized by being not divisible by $t=4$.
  \item This exhausts all $\ell$ dividing root lengths for affine Lie algebras.
    For all other values of $\ell$ we have the generic case $\ell\nmid
    (\alpha,\alpha)$ for all real roots $\alpha$.
  \end{itemize}
\end{proof}

We now conclude the proof of Lemma \nref{lm_primitive}:\\

The table in the statement is taken by dualizing the table in Theorem
  \nref{thm_longAffine} hence by Lemma \nref{lm_lt} the roots
  $\alpha_i^{(0)}\in\left((\Pi^\vee)^{t}\right)^\vee \subset
  \left((\Delta^\vee)^{t}\right)^\vee$ in the statement fulfill
  $q^{(\alpha_i^{(0)},\alpha_i^{(0)})}=1$. We wish to show that (in contrast to
  other roots $\alpha\in\left((\Delta^\vee)^{t}\right)^\vee$) the
  $\alpha_i^{(0)}$ give rise to primitive elements Now Lemma
  \nref{lm_reflectionPrimitive} that any reflection $T_{j,1}''(x)$ of a
  primitive element $x$ (especially $x=E_{\alpha_i}$) on simple roots $\alpha_j$
  with $q^{(\alpha_j,\alpha_j)}=1$. The last condition characterizes the
  $\alpha_j\not\in \left((\Delta^\vee)^{t}\right)^\vee$ or dually
  $\alpha_j^\vee\not\in (\Delta^\vee)^{t}$. We now check that the choices
  $\alpha_i^{(0)},\alpha_i^{(0)\;\vee}$ in the statement have all the property
  that they contain only a unique simple root $\alpha_k^\vee\in
  (\Delta^\vee)^{t}$ and only with multiplicity one (this has already been
  noticed in Theorem \nref{thm_longAffine}). The
  $\alpha_i^{(0)},\alpha_i^{(0)\;\vee}$ can hence be obtained by
  iterated reflection of $E_{\alpha_k}$ only on simple roots $\alpha_j\not\in
  \left((\Delta^\vee)^{t}\right)^\vee$ and hence $E_{\alpha_i^{(0)}}$
  is primitive as asserted.
\end{proof}

\section{Exotic cases}\label{sec_exotic}

We have established in Lemma \nref{lm_primitive} a set of primitive elements
$E_{\alpha_i^{(0)}}\in U_q^{\L}(\g)^+$ associated to the dual of the simple
roots in a subsystem $(\Delta^\vee)^{t}\subset \Delta$. However, it
is neither clear that these are all primitive elements nor that they indeed
generate an affine quantum group of type $\left((\Delta^\vee)^{t}\right)^\vee$.

\begin{example}\label{exm_G2}
In \cite{Len14c} the author has already determined for finite root systems the
exotic case $u_{\sqrt{-1}}^\L(G_2)$ which contains all root vectors but is not
generated by $E_{\alpha_1},E_{\alpha_2}$. Rather, these primitive elements only
generate a quantum group of type $A_2$ and $E_{\alpha_{112}}$ is a new
primitive generator. Altogether we found $u_{\sqrt{-1}}^\L(G_2)^+\cong
u_{\sqrt{-1}}^\L(A_3)^+$.
\end{example}

In this case we were explicitly checking the braiding matrix for
$E_{\alpha_1},E_{\alpha_2}$ against Heckenberger's classification of
finite-dimensional Nichols algebras \cite{Heck09}.\\

We now turn to the potentially exotic cases in Lemma \nref{lm_cases}
which do not involve odd cycles:
\begin{center}
\begin{tabular}{cc}
      $\g,\Delta$ & $\ell$ \\
      \hline\hline
      $A^{(2)}_{2},A^{(2)}_{2n}$ & $3,6$\\
      $G_2^{(1)},D_4^{(3)}$ & $4$\\
\end{tabular}
\end{center}
These were all cases which were not degenerate in the sense that 
$\alpha_i^{(0)}=\alpha_i$. We compute in each case the braiding matrix of the
braided vector space $M$ spanned by the $E_{\alpha_i}$ as we did for the finite
root systems in \cite{Len14c} Thm 5.4 and determine the Cartan matrix for the
Nichols algebra $\B(M)$ using \cite{Heck06} Sec. 3. Mostly we find that $\B(M)$
is significantly smaller than $u_q^\L$ (in terms of real roots) and we thus
successively add new primitive elements (some using Lemma
\nref{lm_reflectionPrimitive}, some by guess-and-check) until we account for
all roots $\alpha$ with $\ell_\alpha\neq 1$. We treat the cases in order of
increasing difficulty:\\

\subsection{Case 
\texorpdfstring{$G_2^{(1)}$ at $\ell=4$}{G2(1 at l=4)}
}~\\

For $\g=G_2^{(1)},\ell=4$ the braiding matrix is
  $$\begin{pmatrix}
    q^6 & q^{-3} & 1\\
    q^{-3} & q^6 & q^{-3} \\
    1 & q^{-3} & q^2  
  \end{pmatrix}
  =\begin{pmatrix}
    \bar{q}^2 & \bar{q}^{-1} & 1\\
    \bar{q}^{-1} & \bar{q}^2 & \bar{q}^{-1}\\
    1 & \bar{q}^{-1} & \bar{q}^2 
  \end{pmatrix}$$
which is the standard braiding matrix $q_{ij}=q'^{(\alpha_i,\alpha_j)}$ for
$A_3,q'$ with $q':=\bar{q}$. Especially (as in all exotic cases) the
$E_{\alpha_i}$ do not generate the expected root system $G_2^{(1)}$.\\

We wish to determine more primitives: We have already seen explicitly in
\cite{Len14c} Thm. 5.4 that for $\g=G_2,\ell=4$ (with indices $1,2$ switched)
the element $$E_{221}:=-q^2(E_1E_2^{(2)}-q^{-6}E_2^{(2)}E_1)-qE_{12}E_2$$
is a primitive and not in the subalgebra generated by
$E_{\alpha_1},E_{\alpha_2}$. The new braided vector space spanned by
$E_{\alpha_0},E_{\alpha_1},E_{\alpha_2},E_{\alpha_1+2\alpha_2}$ then the new
extended braiding matrix is easily calculated:
\begin{align*}
  q^{(\alpha_0,\alpha_1+2\alpha_2)}
  &=q^{-3}=q\\
  q^{(\alpha_1,\alpha_1+2\alpha_2)}
  &=q^{6-6}=+1\\
  q^{(\alpha_2,\alpha_1+2\alpha_2)}
  &=q^{-3+4}=q\\
  q^{(\alpha_1+2\alpha_2,\alpha_1+2\alpha_2)}
  &=q^{6-24+8}=q^2
\end{align*}
$$\begin{pmatrix}
    \bar{q}^2 & \bar{q}^{-1} & 1 & \bar{q}^{-1}\\
    \bar{q}^{-1} & \bar{q}^2 & \bar{q}^{-1} & 1\\
    1 & \bar{q}^{-1} & \bar{q}^2 & \bar{q}^{-1}\\
    \bar{q}^{-1} & 1 &  \bar{q}^{-1} & \bar{q}^2
\end{pmatrix}$$
This is a standard braiding matrix $q_{ij}=q'^{(\alpha_i,\alpha_j)}$ for
the affine Lie algebra $A_3^{(1)},q'$ with $q':=\bar{q}$. Especially the $6$
roots generated by $E_{\alpha_1},E_{\alpha_2},E_{\alpha_1+2\alpha_2}$ account
for all roots in $G_2$. We furthermore check that indeed under this
correspondence
\begin{align*}
  \delta_{A_3^{(1)}}
  &=\alpha_0+\theta_{A_3}\\
  &=\alpha_0+\alpha_1+\alpha_2+(\alpha_1+2\alpha_2)\\
  &=\alpha_0+2\alpha_1+3\alpha_2\\
  &=\alpha_0+\theta_{G_2}=\delta_{G_2^{(1)}}
\end{align*}
This also shows that the isotropic roots in $G_2^{(1)},A_3^{(1)}$ coincide.
So the primitive elements
$E_{\alpha_0},E_{\alpha_1},E_{\alpha_2},E_{\alpha_1+2\alpha_2}$ generate an
affine quantum group of type $A_3^{(1)}$ which contains all roots of
$G_2^{(1)}$.\\

\subsection{Case
\texorpdfstring{$A_2^{(2)}$ at $\ell=3,6$}{A2(2) at l=3,6}
}~\\

For $\g=A^{(2)}_{2}$ the braiding matrix is for $\ell=3$
  $$\begin{pmatrix}
    q^2 & q^{-4}\\
    q^{-4} & q^8
  \end{pmatrix}
  =\begin{pmatrix}
    q^2 & q^{-1}\\
    q^{-1} & q^2
  \end{pmatrix}$$
respectively for $\ell=6$
  $$\begin{pmatrix}
    q^2 & q^{-4}\\
    q^{-4} & q^8
  \end{pmatrix}
  =\begin{pmatrix}
    (-q)^2 & (-q)^{-1}\\
    (-q)^{-1} & (-q)^2
  \end{pmatrix}$$
The braiding matrix for $\ell=3$ is the standard braiding matrix
$q_{ij}=q^{(\alpha_i,\alpha_j)}$ for $A_2,q$. For $\ell=6$ the braiding matrix
is the standard braiding matrix $q_{ij}=q'^{(\alpha_i,\alpha_j)}$ for
$A_2,q'$ with $q':=-q$ again of order $3$.\\
In both cases this does not generated the root system $A_2^{(2)}$.\\

We wish to determine more primitives and treat both cases simultaneous using
$\epsilon:=q^3=\pm 1$.Similar to the $G_2^{(1)}$-case we find a primitive
element $E_{0001}:=E_0^{(3)}E_1-E_1E_0^{(3)}+\cdots$ in degree
$3\alpha_0+\alpha_1$, which cannot be obtained by commutators of $E_0,E_1$. \\
The new braided vector space spanned by
$E_{\alpha_0},E_{\alpha_1},E_{3\alpha_0+\alpha_1}$, then the new
extended braiding matrices for $\ell=3$ resp. $\ell=6$ are easily calculated:
\begin{align*}
  q^{(\alpha_0,3\alpha_0+\alpha_1)}
  &=q^{6-4}=q^2=q^{-1}\mbox{ resp. }(-q)^{-1}\\
  q^{(\alpha_1,3\alpha_0+\alpha_1)}
  &=q^{-12+8}=q^{-4}=q^{-1}\mbox{ resp. }(-q)^{-1} \\
  q^{(3\alpha_0+\alpha_1,3\alpha_0+\alpha_1)}
  &=q^{18-24+8}=q^2
\end{align*}
  $$\begin{pmatrix}
    q^2 & q^{-1} & q^{-1}\\
    q^{-1} & q^2 & q^{-1}\\
    q^{-1} & q^{-1} & q^2 
  \end{pmatrix}\qquad 
  \begin{pmatrix}
    (-q)^2 & (-q)^{-1} & (-q)^{-1} \\
    (-q)^{-1} & (-q)^2 & (-q)^{-1} \\
    (-q)^{-1} & (-q)^{-1} & (-q)^2 
  \end{pmatrix}$$
These are the standard braiding matrices $q_{ij}=q'^{(\alpha_i,\alpha_j)}$ of
type $A_2^{(1)}$ with $q':=q$ for $\ell=3$ resp. $q'=-q$ for $\ell=6$.\\

We now want to convince ourselves that the affine quantum group of type
$A_2^{(1)}$ generated by the three primitives
$E_{\alpha_0},E_{\alpha_1},E_{3\alpha_0+\alpha_1}$ in $A_2^{(2)}$ does indeed
generate the full root system of $\Delta=A_2^{(2)}$, which is by Section
\nref{subsec_affine}
$$\Delta^{re}= \left(\{\pm\alpha_1\}+2\delta_{A_2^{(2)}}\Z\right)
  \;\cup\;
\left(\frac{1}{2}(\{\pm\alpha_1\}+\delta_{A_2^{(2)}})+\delta_{A_2^{(2)}}
\Z\right)$$
We first compare $\delta_{A_2^{(2)}},\delta_{A_2^{(1)}}$ under this
correspondence:
\begin{align*}
  \delta_{A_2^{(1)}}
  &=\alpha_0+\theta_{A_2}\\
  &=\alpha_0+\alpha_1+(3\alpha_0+\alpha_1)\\
  &=2(2\alpha_0+\alpha_1)
  =2\delta_{A_2^{(2)}}
\end{align*}
The verify that the real roots are in bijection it hence suffices to find
matching fundamental domains of the action $+2\delta_{A_2^{(2)}}$ on both root
systems. The following works:
\begin{align*}
  &\{\pm\alpha_1\}
  \;\cup\; \left(\frac{1}{2}(\{\pm\alpha_1\}+\delta_{A_2^{(2)}})
  +\delta_{A_2^{(2)}}\right)
  \;\cup\; \left(\frac{1}{2}(\{\pm\alpha_1\}+\delta_{A_2^{(2)}})
  -2\delta_{A_2^{(2)}}\right)\\
  &=\{\pm\alpha_1\}\;\cup\;\{\alpha_{00011},\alpha_{0001}\}
  \;\cup\;\{-\alpha_{0001},-\alpha_{00011}\}\\
  &=\{\pm\alpha_1,\pm\alpha_{0001},\pm\alpha_{00011} \} 
\end{align*}
As a side remark, note this gets significantly nicer if one
rotates $A_2^{(1)}$ such that $\alpha_0'=\alpha_{0001}$.
Hence the real roots of $A_2^{(1)}$ and $A_2^{(2)}$ coincide under our
correspondence.\\

Note however this \emph{fails} for the isotropic roots: In
$A_2^{(2)}$ we have isotropic roots $\delta_{A_2^{(2)}}$ with 
multiplicities $1$, while in $A_2^{(1)}$ we have the 
isotropic roots $\delta_{A_2^{(1)}}=2m\delta_{A_2^{(2)}}$ with multiplicity
$2$. The author does not have an explanation for this.\\

\subsection{Case
\texorpdfstring{$D_4^{(3)}$ at $\ell=4$}{D4(3) at l=4}
}~\\

For $\g=D_4^{(3)},\ell=4$ the braiding matrix is
  $$\begin{pmatrix}
    q^2 & q^{-1} & 1\\
    q^{-1} & q^2 & q^{-3} \\
    1 & q^{-3} & q^6  
  \end{pmatrix}
  =\begin{pmatrix}
    q^2 & q^{-1} & 1\\
    q^{-1} & q^2 & \bar{q}^{-1}\\
    1 & \bar{q}^{-1} & q^2 
  \end{pmatrix}$$
This is not a braiding matrix of the form $q_{ij}=q'^{(\alpha_i,\alpha_j)}$, but
it is nevertheless of type $A_3$ by \cite{Heck06}
with $q_{ii}=q^2=-1=(q_{ij}q_{ji})^{-1}$ for $i,j$ adjacent. It could be
rewritten as a Doi twist
of $A_3,q$ or $A_3,\bar{q}$. Especially (as in all exotic cases) the
$E_{\alpha_i}$ do not generate the expected root system $D_4^{(3)}$.\\

We wish to determine more primitives: We have already seen explicitly in
\cite{Len14c} Thm. 5.4 that for $\g=G_2,\ell=4$
the element $$E_{112}:=-q^2(E_2E_1^{(2)}-q^{-6}E_1^{(2)}E_2)-qE_{12}E_1$$
is a primitive and not in the subalgebra generated by
$E_{\alpha_1},E_{\alpha_2}$. The new braided vector space spanned by
$E_{\alpha_0},E_{\alpha_1},E_{\alpha_2},E_{2\alpha_1+\alpha_2}$ then the new
extended braiding matrix is easily calculated:
\begin{align*}
  q^{(\alpha_0,2\alpha_1+\alpha_2)}
  &=q^{-2}=-1\\
  q^{(\alpha_1,2\alpha_1+\alpha_2)}
  &=q^{4-3}=q\\
  q^{(\alpha_2,2\alpha_1+\alpha_2)}
  &=q^{-6+6}=1\\
  q^{(2\alpha_1+\alpha_2,2\alpha_1+\alpha_2)}
  &=q^{8-12+6}=q^2
\end{align*}
$$\begin{pmatrix}
    q^2 & q^{-1} & 1 & -1 \\
    q^{-1} & q^2 & \bar{q}^{-1} & \bar{q}^{-1}\\
    1 & \bar{q}^{-1} & q^2 & 1 \\
    -1 & \bar{q}^{-1} & 1 & q^2
  \end{pmatrix}$$
This is not a braiding matrix of the form $q_{ij}=q'^{(\alpha_i,\alpha_j)}$, but
it is nevertheless of type $D_4$ with center node $\alpha_1$ by \cite{Heck06}
with $q_{ii}=q^2=-1=(q_{ij}q_{ji})^{-1}$ for $i,j$ adjacent (note especially
now two non-adjacent nodes $\alpha_0,2\alpha_{1}+\alpha_2$ anticommute). It
could be rewritten as a Doi twist of $D_4,q$ or $D_4,\bar{q}$. Especially still
the $E_{\alpha_0},E_{\alpha_1},E_{\alpha_2},E_{2\alpha_1+\alpha_2}$ do not
generate the expected root system $D_4^{(3)}$.\\

We wish to determine more primitives: We observe that $E_0,E_{112}$ anticommute 
in the $D_4$ subalgebra established above even though
$(\alpha_0,\alpha_{112})=-2$ in the root system $D_4^{(3)}$ (this is a typical
effect in exotic cases). We get hence from lemma \nref{lm_reflectionPrimitive}
that the reflection of $E_{112}$ on $\alpha_0$ is a primitive element in degree
$2\alpha_0+2\alpha_1+\alpha_2$. The new braided vector space spanned by
$E_{\alpha_0},E_{\alpha_1},E_{\alpha_2},E_{2\alpha_1+\alpha_2},E_{
2\alpha_0+2\alpha_1+\alpha_2}$, then the new
extended braiding matrix is easily calculated:
\begin{align*}
  q^{(\alpha_0,2\alpha_0+2\alpha_1+\alpha_2)}
  &=q^{4-2}=q^2=-1\\
  q^{(\alpha_1,2\alpha_0+2\alpha_1+\alpha_2)}
  &=q^{-2+4-3}=q^{-1}\\
  q^{(\alpha_2,2\alpha_0+2\alpha_1+\alpha_2)}
  &=q^{-6+6}=1\\
  q^{(2\alpha_1+\alpha_2,2\alpha_0+2\alpha_1+\alpha_2)}
  &=q^{-4+4}=1\\
  q^{(2\alpha_0+2\alpha_1+\alpha_2,2\alpha_0+2\alpha_1+\alpha_2)}
  &=q^{8-8+8-12+6}=q^2
\end{align*}
$$\begin{pmatrix}
    q^2 & q^{-1} & 1 & -1 & -1\\
    q^{-1} & q^2 & \bar{q}^{-1} & \bar{q}^{-1} & q^{-1}\\
    1 & \bar{q}^{-1} & q^2 & 1 & 1\\
    -1 & \bar{q}^{-1} & 1 & q^2 & 1\\
    -1 & q^{-1} & 1 & 1 & q^2
  \end{pmatrix}$$
This is not a braiding matrix of the form $q_{ij}=q'^{(\alpha_i,\alpha_j)}$, but
it is nevertheless of type $D_4^{(1)}$ with center node $\alpha_1$ by
\cite{Heck06} with $q_{ii}=q^2=-1=(q_{ij}q_{ji})^{-1}$ for $i,j$ adjacent.\\

We now want to convince ourselves that the affine quantum group of type
$D_4^{(1)}$ generated by the four primitives
$E_{\alpha_0},E_{\alpha_1},E_{\alpha_2},E_{2\alpha_1+\alpha_2},E_{
2\alpha_0+2\alpha_1+\alpha_2}$ in $D_4^{(3)}$ does indeed generate the full
root system of $\Delta=D_4^{(3)}$, which is by Section \nref{subsec_affine}
$$\Delta^{re}=\left(\bar{\Delta}^{short}+\delta\Z\right)
  \;\cup\; \left(\bar{\Delta}^{long}+\delta k\Z\right)$$
With $\bar{\Delta}=G_2$ and
$\delta_{D_4^{(3)}}=\alpha_0+\theta^{short}_{G_2}=\alpha_{0112}$ this is
explicitly 
$$\Delta^{re}=\left(\left\{\pm
\alpha_1,\pm\alpha_{12},\pm\alpha_{112}\right\}+\alpha_{0112}\Z
\right)\;\cup\;
\left(\left\{\alpha_2,\pm\alpha_{1112},\pm\alpha_{11122}\right\}
+3\alpha_{0112}\Z\right)$$
On the other hand we have for the root system $\Delta'=D_4^{(1)}$ that
$\bar{\Delta}'=D_4$ and under the correspondence 
\begin{align*}
  \delta_{D_4^{(1)}}
  &=\alpha_0'+\theta^{long}_{D_4}\\
  &=\alpha_0'+(2\alpha_1'+\alpha_2'+\alpha_3'+\alpha_4')\\
  &=\alpha_0+(2\alpha_1+\alpha_2+\alpha_{112}+\alpha_{00112})\\
  &=3\alpha_0+6\alpha_1+3\alpha_2
  =3\delta_{D_4^{(3)}}
\end{align*}
We hence have to convince ourselves whether some fundamental domain for the
action $+3\delta_{D_4^{(3)}}$ coincides, say
\begin{align*}
&\bar{\Delta}^{short}
\;\cup\;\left(\bar{\Delta}^{short}+\delta_{D_4^{(3)}}\right)
\;\cup\;\left(\bar{\Delta}^{short}+2\delta_{D_4^{(3)}}\right)
\;\cup\;\bar{\Delta}^{long}
\;\stackrel{?}{=}\;\bar{\Delta}'
\end{align*}
This is not completely true, but a slightly more complicated fundamental domain
on the right-hand side suffices: As in the
finite $G_2$ example \nref{exm_G2}, the roots in $\bar{\Delta}=G_2$ are in
bijection with the $A_3$ subsystem in $\bar{\Delta}'=D_4$ generated by
$\alpha_1,\alpha_2,\alpha_{112}$. The three roots of
$\bar{\Delta^{short}}+2\delta_{D_4^{(3)}}$ are the three larger roots in $D_4$
containing $\alpha_4'$, namely
\begin{align*}
  \alpha_1+2\alpha_{0112}
  &=\alpha_1+\alpha_{112}+\alpha_{00112}
  =\alpha_1'+\alpha_3'+\alpha_4'\\
  \alpha_{12}+2\alpha_{0112}
  &=\alpha_1+\alpha_2+\alpha_{112}+\alpha_{00112}
  =\alpha_1'+\alpha_2'+\alpha_3'+\alpha_4'\\
  \alpha_{112}+2\alpha_{0112}
  &=2\alpha_1+\alpha_2+\alpha_{112}+\alpha_{00112}
  =2\alpha_1'+\alpha_2'+\alpha_3'+\alpha_4'
\end{align*}
The three roots of $\bar{\Delta^{short}}+2\delta_{D_4^{(3)}}$ however
correspond to \emph{shifted, negative} versions of the remaining roots
$\alpha_4,\alpha_1'+\alpha_4,\alpha_1'+\alpha_2'+\alpha_4'$ in $D_4$, namely:
\begin{align*}
  \alpha_1+\alpha_{0112}
  &=-\alpha_1-\alpha_{2}-\alpha_{00112}+3\alpha_{0112}
  =-(\alpha_1'+\alpha_2'+\alpha_4')+\delta_{D_4^{(1)}}\\
  \alpha_{12}+\alpha_{0112}
   &=-\alpha_1-\alpha_{00112}+3\alpha_{0112}
  =-(\alpha_1'+\alpha_4')+\delta_{D_4^{(1)}}\\
  \alpha_{112}+\alpha_{0112}
   &=-\alpha_{00112}+3\alpha_{0112}
  =-\alpha_4'+\delta_{D_4^{(1)}}
\end{align*}
Hence the real roots of $D_4^{(1)}$ and $D_4^{(3)}$ coincide under our
correspondence.\\

Note however this \emph{fails} for the isotropic roots: In
$D_4^{(3)}$ we have isotropic roots $m\delta_{D_4^{(3)}}$ with 
multiplicities $1,1,2$ depending on $m\mod 3$, while in $D_4^{(3)}$ we have
isotropic roots $m\delta_{D_4^{(1)}}=3m\delta_{D_4^{(3)}}$ with multiplicity
$4$. The author does not have an explanation for this.\\

\subsection{Case
\texorpdfstring{$A_{2n}^{(2)}$ at $\ell=3,6$}{A2n(2) at l=3,6}
}~\\

For $\g=A^{(2)}_{2n}$ with $n\geq 2$ the braiding matrix is
$$\begin{pmatrix}
     q^2 & q^{-2} & 1 & \cdots & 1 & 1 \\
     q^{-2} & q^4 & q^{-2} & \cdots & \cdots & 1 \\
     1 & q^{-2} &  q^4  & q^{-2} & \cdots & \cdots \\
     \cdots & \cdots & q^{-2} & \cdots &\cdots & \cdots \\
     1 & \cdots & \cdots & \cdots & q^4 & q^{-4} \\
     1 & 1 & \cdots & \cdots & q^{-4} & q^8 \\
   \end{pmatrix}$$
We will verify that the root system and Weyl group in these cases remain
$A^{(2)}_{2n}$. This is slightly inconvenient, because we cannot present
$u_q^\L(\g)$ as isomorphic to a different $u_{q'}^\L(\g')$ with $\g',q'$
fulfilling Lusztig's non-degeneracy assumptions in Lemma
\nref{lm_LusztigGeneric}. On the other hand it just means this case is not
really degenerate (in contrast to $n=1$ above).\\

By \cite{Heck06} Sec. 3 the Cartan matrix of the Nichols algebra of a braided
vector space is given by $a_{ii}=2$ and $a_{ij}=-m_{ij}$ with
$$m_{ij}=\min\left\{m\;\mid\;(m+1)_{q_{ii}}(q_{ii}^mq_{ij}q_{ji}-1)=0\right\}$$ 
We have for all $i$ that $q_{ii}=q^2,q^4,q^8$ has order $3$ for $\ell=3,6$,
hence $(m+1)_{q_{ii}}=0$ for $m\geq 2$. The term $q_{ij}q_{ji}$
is $=1$ for non-adjacent $i,j$ and $q^{-4},q^{-8}$ with again order $3$ for
adjacent $i,j$. Hence $q_{ii}^mq_{ij}q_{ji}-1=0$ for $m=0$ and hence
$m_{ij}=0$ holds precisely for non-adjacent $i,j$. It remains to determine
when $q_{ii}^mq_{ij}q_{ji}-1=0$ for $m=1$ i.e. $q_{ii}(q_{ij}q_{ji})=1$.
Checking our $2\times 3$ cases we see this is only possible for
$q_{ii}=q^4,q_{ij}q_{ji}=q^{-4}$ and $q_{ii}=q^8,q_{ij}q_{ji}=q^{-8}$ and
$q_{ii}=q^2,q_{ij}q_{ji}=q^{-8},\ell=6$. The last case does not appear for
$n\geq 2$, the first case is quite frequent and the second case appears at the
last node. Altogether we find the Cartan matrix for $\ell=3,6$, and hence the
Weyl group, is unchanged:

$$a_{ij}=\begin{pmatrix}
     2 & -2 & 1 & \cdots & 0 & 0 \\
     -1 & 2 & -1 & \cdots & \cdots & 0 \\
     1 & -1 &  2  & -1 & \cdots & \cdots \\
     \cdots & \cdots & -1 & \cdots &\cdots & \cdots \\
     0 & \cdots & \cdots & \cdots & 2 & -2 \\
     0 & 0 & \cdots & \cdots & -1 & 2 \\
   \end{pmatrix}$$

\section{Deaffinized cases involving 
\texorpdfstring{$A_1^{(1)}$ at $\ell=4$}{A1(1) at l=4}
}\label{sec_A11}

Curiously, the case $\g=A_1^{(1)}=\hat{\sl}_2,\ell=4$ is in some sense the most
exceptional case, namely the only one not leading to a (maybe different)
affine Lie algebra. This is due to the fact, that all roots are short (hence
$u_q^\L$ is generated by all $E_{\alpha_i}$), but the Cartan matrix has entries
$\pm 2$ (hence the braiding matrix degenerates to $A_1\times A_1$), but the
Hopf algebra is not coradically graded (hence there are nontrivial lifting
relations). We shall analyze this case in the following section; not that this
in contained in several cases of larger rank in Lemma \nref{lm_primitive}.\\

We first describe the braiding matrix of $\g=A_1^{(1)}$ for $\ell=4$:
$$\begin{pmatrix}
    q^2 & q^{-2} \\
    q^{-2} & q^2 
  \end{pmatrix}
=\begin{pmatrix}
    -1 & -1 \\
    -1 & -1 
  \end{pmatrix}$$
This is not a braiding matrix of the form $q_{ij}=q'^{(\alpha_i,\alpha_j)}$, but
it is nevertheless of type $A_1\times A_1$ by
\cite{Heck06} with $q_{ii}=-1$ and $q_{ij}q_{ji}=-1$. This implies easily (e.g.
Lemma \nref{lm_primitive}a) that the
element in the isotropic degree $\delta=\alpha_0+\alpha_0$
$$E_{\alpha_0+\alpha_1}:=
E_{\alpha_0}E_{\alpha_1}+E_{\alpha_1}E_{\alpha_0}$$
is a new primitive element. For a coradically graded Hopf algebra (e.g. the
Nichols algebra) this would imply the commutator is zero as expected for
$A_1\times A_1$; the same holds for all other commutators between higher root
vectors and one might expect our $A_1^{(1)}$ to degenerate to $A_1^{\times
\infty}$. We will see that this is not true: While indeed $\gr(u_q^\L(\g))$
will be of type $A_1^{\times\infty}$ (one copy for each root), \emph{some}
nontrivial commutator relations will survive in $u_q^\L(\g)$.\\

To pursue our study, we need to calculate in an explicit PBW-basis by Beck for
quantum groups of \emph{untwisted} affine Lie algebras. The given alternative
set of generators and relations for $U^\L_q(\g)$ is essentially due to Drinfel'd
and emphasizes the construction of $\g$ from a loop algebra (see \cite{Beck94}
Thm 4.7):\\

Let $\bar{I}$ again denote the index set of the finite Lie algebra associated
to $\g$, then we have generators (``root vectors'') $x_{\alpha_i,j}^\pm$ in
degree $\pm\alpha_i$ and $h_{\alpha_i,k},k\neq 0$ in degree $k\delta$ and
$K_i^{\pm1}, C^{\pm\frac{1}{2}}, D$ in degree $0$. We do not give al relations
here. Now \cite{Beck94} Prop. 6.1 states that these elements form a PBW-basis
for $U_q^{\Q(q)}(\g)$, but also for the Lusztig integral form 
$U_q^{\Z[q,q^{-1}],\L}(\g)$ and hence the specialization $U_q^\L(\g)$.\\

Relation $(5)$ states that $[x_{\alpha_i,k}^+,x_{\alpha_i,l}^-]\neq 0$ in
any specialization $\ell\neq 1,2$; especially in our case
$$[x_{\alpha_1,0}^+,x_{\alpha_1,1}^-]=
C^{-\frac{1}{2}}K_{\alpha_1}\cdot h_{\alpha_1,1}
+C^{\frac{1}{2}}K_{\alpha_1}^{-1}\cdot h_{\alpha_1,1}$$
 On the other hand relations
$(2),(3)$ states that the commutators $[h_{\alpha_i,k},h_{\alpha_i,-k}]$ and
$[h_{\alpha_i,k},x_{\alpha_i,l}^\pm]$ contain a term $[ka_{ij}]_{q_i}$. This
implies in our case $a_{ij}=-2$ and $\ord(q_i)=\ord(q)=\ell=4$ that all
$h_{\alpha_i,k}$ are central in the specialization. Hence the algebra generated
by $x_{\alpha_1,0}^+,x_{\alpha_1,1}^-=E_1,E_0$ is finite-dimensional and has a
root system of type $A_2$. It is \emph{not coradically graded}, rather the
respective graded Nichols algebra is of type $A_1^{\times 3}$.\\ 

By definition, the reflections of $x_{\alpha_1,0}^+,x_{\alpha_1,1}^-$ are the
higher root vectors $x_{\alpha_1,k}^+,x_{\alpha_1,l}^-$. The formula in Lemma
\nref{lm_primitive} easily shows that they are not primitive. Hence
$u_q^\L{A_1^{1}}$ is \emph{not generated in degree $1$}. We have 
$$u_q^\L(A_1^{1})^+/\langle E_1,E_0\rangle\cong u_q^\L(A_1^{1})^+$$
and hence $u_q^\L(A_1^{1})^+$ is an infinite tower extension of $u_q^\L(A_2)$'s;
the higher order lifting relations can be read off directly from the cited
relations.

\section{Open Questions}\label{sec_questions}

We finally give some open questions that the author would find interesting:\\

Regarding the subsystem of long roots:
\begin{problem}
  Is there an easier systematic reason why the subsystem of long roots always
  returns the affinized subsystem of long roots of the corresponding finite Lie
  algebra - both for twisted and untwisted type?
\end{problem}

Regarding the limitations of this article:
\begin{problem}
  Explicitly determine the algebras $u_q^\L(\g)$ in terms of
  some $u_{q'}(\g^{(0)})$.
  \begin{itemize}
    \item Is the surjection $u_q^\L(\g)^+\to \B(M)=u_{q'}(\g^{(0)})^+$ also for
    twisted affine $\g$ an isomorphism? By construction this is true on the
    level of roots, but there seems at present to be no PBW-basis available in
    this case.
    \item Are there nontrivial liftings in the exotic cases? It seems this would
    require explicit calculations. 
    \item Present $u_q^\L(\g)$ as a quotient of the Drinfel'd double of
    $u_q^\L(\g)^+$.
  \end{itemize}
  More generally, describe the affine quantum groups in terms of coradical
  filtration, Nichols algebra and a Drinfel'd double construction. Can the
  representations be related to the Yetter-Drinfel'd modules of the Borel part?
  Can one clarify the impact of the second condition (odd cycles) in
  \cite{Lusz94}?
\end{problem}

Regarding the construction of a Frobenius homomorphism:
\begin{problem}
  The Hopf subalgebras $u_q^\L(\g)$ constructed in this article are a good
  candidate for the kernel of a \emph{Frobenius homomorphism}. In cases where we
  have a Frobenius homomorphism, namely \cite{Lusz94} for sufficiently large $q$
  and \cite{Len14c} for the remaining small $q$ and finite root system, it
  coincides. Is it possible to extend the approach in \cite{Len14c}, namely
  showing normality of this $u_q^\L(\g)$ and then analyzing the quotient, to
  construct a Frobenius homomorphism for affine quantum groups for arbitrary
  $q$? It is to be expected that the quotient is \emph{not} simply the
  universal enveloping of the same Lie algebra.
\end{problem}

Regarding an observation that puzzled the author during the work on this
article:
\begin{problem}
  While the subalgebras $u_q^\L(\g)^+$ exhaust by construction precisely the
  real roots, the \emph{multiplicities} of the isotropic roots sometimes do not
  coincide:
  \begin{itemize}
  \item  We have constructed a map $u_q^\L(D^{(1)}_n)^+\to
  u_q^\L(A_{2n-1}^{(2)})^+$ with coinciding $\delta$. However, for
  $A_{2n-1}^{(2)}$ the multiplicities are $1,1,2,1,1,2,\ldots$ while for
  $D_n^{(1)}$ they are as usual $2,2,2,2,\ldots$. A similar effect occurs for  
  $E_6^{(2)},D_4^{(3)}$. This could mean two  both unlikely things: Either the 
  map is not injective (but this is absurd for  a map from a Nichols algebra)
  or the isotropic root multiplicities in $U_q^\L(A_{2n-1}^{(2)})$ are not as
  for $A_{2n-1}^{(2)}$. This would be surprising, but note that the root system 
  itself does not determine the multiplicities, they are rather calculate from 
  the twisted presentation.
  \item In the exotic cases there seems a systematic behaviour: Take as example
  $D_4^{(3)}$ at $\ell=4$ which has a $D_4^{(1)}$ root system with
  $3\delta_{D_4^{(3)}}=\delta_{D_4^{(1)}}$. The isotropic root multiplicities
  are $1,1,2,1,1,3,\ldots$ respectively $0,0,4,0,0,4,\ldots$ - they seem just
  to be shifted, the same hold for the other exotic cases. Now: One can easily 
  obtain primitive elements (with trivial self-braiding) for the first
  isotropic roots in $D_4^{(3)}$, especially they are not in $D_4^{(1)}$, and
  they turn out to still reside ``inside'' the first set of commutator
  relations. Intuitively, do they extend $D_4^{(1)}$ from below? Are their
  powers or other commutators than the elements in degree $3\delta$ we observe
  in $D_4^{(1)}$?
  \end{itemize}
\end{problem}
 
Regarding parabolic subalgebras:
\begin{problem}
  Let $J\subset I$ a subset of simple roots, then we have a parabolic
  subalgebra $\g_J\subset \g$ of finite type and $U_q^\L(\g_J)\subset
  U_q^\L(\g)$ and $u_q^\L(\g_J)\subset u_q^\L(\g)$. By the  Radford
  projection theorem we can view  $u_q^\L(\g)^+$ as a Hopf algebra in  the
  category of $u_q^\L(\g_J)^+$-Yetter-Drinfel'd modules (hence morally a 
  module over $u_q^\L(\g_J)$, which is finite-dimensional). It would be
  extremely interesting to study this source of infinite-dimensional
  representations of finite quantum groups. The representation should decompose
  into infinitely many finite-dimensional modules associated to the finite root
  strings in $\g$.
\end{problem}

Regarding applications:
\begin{problem}
  Are there certain small values for the coupling constants in affine Toda
  theory such that the symmetry is described by the quantum affine algebras
  constructed here? Is this degeneracy physically visible? (as in the finite
  case $B_n,\ell=4$ for $n$ symplectic Fermions in \cite{Len14c})
\end{problem}
\begin{problem}
  It would be extremely interesting to study quantum groups for hyperbolic Lie
  algebras, at least in examples, by using techniques from Nichols algebras.
  Are there again ``deaffinized'' cases which decompose into an infinite tower
  of finite Lie algebras (or something related)?
\end{problem}
\begin{problem}
  It is a curious observation, that the degenerate cases $\g,q$ in \cite{Len14c}
  for $q$ finite (having h a root system different than $\g$) seem to match
  the list of Nichols algebras with root  system $\g$ and self-braiding $q$ over
  nonabelian groups constructed by the author in \cite{Len14a} as
  diagram-folding of different $u_q(\g')$; as if they  would be somehow
  replaced.\\
  One should use a presentation of affine Lie algebras by folding of other
  affine Lie algebras as indicated in Section \nref{subsec_affine} (this is
  not the same as the construction of twisted Lie algebras) and use this to
  construct quantum affine algebras with nonabelian Cartan part resp. affine
  Nichols algebras over nonabelian groups. Of special interest should be the
  case $A_n^{(1)}$ and $C_n^{(1)}$ which can be folded successively many times.
\end{problem}

%

\end{document}